\documentclass[11pt]{article}

\usepackage{amssymb,amsmath,bm}
\usepackage{amsthm}

\usepackage{textcomp}
\usepackage{enumerate}      
\usepackage{graphicx}        
\usepackage{caption}

\usepackage{tabu}

\usepackage[export]{adjustbox}

\usepackage{mathrsfs}

\usepackage{url}

\usepackage{array}
\newcommand{\PreserveBackslash}[1]{\let\temp=\\#1\let\\=\temp}
\newcolumntype{C}[1]{>{\PreserveBackslash\centering}p{#1}}
\newcolumntype{R}[1]{>{\PreserveBackslash\raggedleft}p{#1}}
\newcolumntype{L}[1]{>{\PreserveBackslash\raggedright}p{#1}}

\renewcommand{\setminus}{{\smallsetminus}}

\newcommand{\cp}[1]{\vcenter{\hbox{#1}}}

\usepackage{amssymb,amsmath,bm}
\usepackage{amsthm}
\usepackage{hyperref}
\usepackage{mathrsfs}
\usepackage{textcomp}
\usepackage{enumerate}
\usepackage{graphicx}
\usepackage{tikz}
\usepackage{verbatim}

\newtheorem{theorem}{Theorem}[section]
\newtheorem{lemma}[theorem]{Lemma}
\newtheorem{proposition}[theorem]{Proposition}
\newtheorem{definition}[theorem]{Definition}
\newtheorem{corollary}[theorem]{Corollary}
\newtheorem{conjecture}[theorem]{Conjecture}

\theoremstyle{remark}
\newtheorem{remark}[theorem]{Remark}

\theoremstyle{remark}

\numberwithin{equation}{section}

\begin{document}
\title{\bf A relative version of the Turaev-Viro invariants and the volume of hyperbolic polyhedral $3$-manifolds}

\author{Tian Yang}

\date{}
\maketitle

\begin{abstract}  We define a relative version of the Turaev-Viro invariants for an ideally triangulated compact $3$-manifold with non-empty boundary and a coloring on the edges, generalizing the Turaev-Viro invariants\,\cite{TV} of the manifold.  We also propose the Volume Conjecture for these invariants  whose asymptotic behavior is related to the volume of the manifold in the hyperbolic polyhedral metric\,\cite{Luo, LY} with singular locus the edges and cone angles determined by the coloring, and prove the conjecture in the case that the cone angles are sufficiently small. This suggests an approach of solving the Volume Conjecture for the Turaev-Viro invariants proposed by Chen-Yang\,\cite{CY} for hyperbolic $3$-manifolds with totally geodesic boundary.
\end{abstract}


\section{Introduction}


Let $M$ be a compact $3$-manifold with non-empty boundary, and let $\mathcal T$ be an ideal triangulation of $M,$ that is, a finite collection  $T=\{\Delta_1,\dots,\Delta_{|T|}\}$ of truncated Euclidean tetrahedra with faces identified in pairs by affine homeomorphisms. We also let $E=\{e_1,\dots,e_{|E|}\}$ be the set of edges of $\mathcal T.$ For a positive integer $r\geqslant 3,$ a \emph{coloring} $\mathbf a$ of $(M,\mathcal T)$ assigns an integer $a_i$ in between $0$ and $r-2$ to the edge $e_i,$ and the coloring $\mathbf a$ is \emph{$r$-admissible} if for any $\{i,j,k\}\subset\{1,\dots, |E|\}$ such that $e_i,$ $e_j$ and $e_k$ are the edges of a face of $\mathcal T,$ 
\begin{enumerate}[(1)]
\item $a_i+a_j-a_k\geqslant 0,$
\item $a_i+a_j+a_k\leqslant 2(r-2),$ and
\item $a_i+a_j+a_k$ is even.
\end{enumerate}

\begin{definition}  Let $r\geqslant 3$ be an integer and let $q$ be a $2r$-th root of unity such that $q^2$ is a primitive $r$-th root of unity. Then the $r$-th relative Turaev-Viro invariant of $(M,\mathcal T)$ with the coloring $\mathbf b=(b_1,\dots,b_{|E|})$ on the edges is defined by
$$\mathrm{TV}_r(M, E, \mathbf b)=\sum_{\mathbf a}\prod_{i=1}^{|E|} \mathrm H(a_i, b_i)\prod_{s=1}^{|T|}\bigg|\begin{matrix} a_{s_1} & a_{s_2}  & a_{s_3} \\  a_{s_4}  &  a_{s_5}  &  a_{s_6}  \end{matrix} \bigg|,$$
where the sum is over all the $r$-admissible colorings  $\mathbf a=(a_1,\dots,a_{|E|})$ of $(M,\mathcal T),$ 
$$\mathrm H(a_i, b_i)=(-1)^{a_i+ b_i}\frac{q^{(a_i+1)(b_i+1)}-q^{-(a_i+1)(b_i+1)}}{q-q^{-1}},$$
$\{a_{s_1},\dots,{a_{s_6}}\}$ are the colors of the edges of the tetrahedron $\Delta_s$ assigned by $\mathbf a$
and $\bigg|\begin{matrix} a_{s_1} & a_{s_2}  & a_{s_3} \\  a_{s_4}  &  a_{s_5}  &  a_{s_6}  \end{matrix} \bigg|$
is the quantum $6j$-symbol of the $6$-tuple $(a_{s_1},\dots,{a_{s_6}}).$ (See Section \ref{6jsymbol}.)

 \end{definition}
 
We note that if $\mathbf b=(0,\dots,0),$ then $\mathrm{TV}_r(M, E, \mathbf b)$ coincides with the Turaev-Viro invariant of $M$\,\cite{TV}. The definition is inspired by the discrete Fourier transforms of the Yokota invariants\,\cite{BAR} and their relationship with the hyperbolic volume of deeply truncated polyhedra\,\cite{BY}. There is a possibility that this definition can be generalized to the setting of modular tensor categories (see \cite[Section 1]{Liu} and references therein). 

Similar to the relationship between the Turaev-Viro invariants of $M$ and the Reshetikhin-Turaev invariants of its double\,\cite{T, R, BP}, the relative Turaev-Viro invariants of $(M,\mathcal T)$ and the relative Reshetikhin-Turaev invariants\,\cite{BHMV, Li} of the double of $M$ is related as follows.
 
 \begin{theorem}\label{=} At $q=e^{\frac{\pi i}{r}},$
 $$\mathrm{TV}_r(M,E, \mathbf b)=\bigg(\frac{2\sin\frac{\pi}{r}}{\sqrt {2r}}\bigg)^{-\chi(M)} \mathrm{RT}_r(D(M), D(E), \mathbf b);$$
 and at $q=e^{\frac{2\pi i}{r}},$
  $$\mathrm{TV}_r(M,E, \mathbf b)=2^{\mathrm{rank}\mathrm H_2(M;\mathbb Z_2)}\bigg(\frac{2\sin\frac{2\pi}{r}}{\sqrt r}\bigg)^{-\chi(M)} \mathrm{RT}_r(D(M), D(E), \mathbf b),$$
 where $\chi(M)$ is the Euler characteristic of $M,$  $D(M)$ is the double of $M$ and $D(E)\subset D(M)$ is the link  consisting of the union of the double of the edges. 
 \end{theorem}
 
Theorem \ref{=} can be proved following the same idea of Roberts \cite{R}. See also \cite{BP} for the case of manifolds with non-empty boundary and \cite{DKY} for the case that $q=e^{\frac{2\pi i}{r}}$ for odd $r.$ For the readers convenience, we include a sketch of the proof of Theorem \ref{=} in Section \ref{TVRT}.
 

Various quantum invariants are expected to contain geometric information of various $3$-dimensional objects. See for example \cite{Ka2, MM, CY, BY, WY2}. For the relative Turaev-Viro invariants, the corresponding geometric object is the hyperbolic polyhedral metric. As defined in \cite{Luo, LY}, a \emph{hyperbolic polyhedral metric} on an ideally triangulated $3$-manifold $(M,\mathcal T)$ is obtained by replacing each tetrahedron in $\mathcal T$ by a truncated hyperideal tetrahedron (see Section \ref{CCS}) and replacing the gluing homeomorphisms between pairs of the faces  by isometries. The \emph{cone angle} at an edge is the sum of the dihedral angles of the truncated hyperideal tetrahedra around the edge. If all the cone angles are equal to $2\pi,$ then the hyperbolic polyhedral metric gives a hyperbolic metric on $M$ with totally geodesic boundary.  In \cite[Theorem 1.2 (b)]{LY}, Luo and the author proved that hyperbolic polyhedral metrics on $(M, \mathcal T )$ are rigid in the sense that they are up to isometry determined by their cone angles.

\begin{conjecture}\label{VC} Let $\{\mathbf b^{(r)}\}$ be a sequence of colorings of $(M,\mathcal T).$ For each $i\in\{1,\dots,|E|\},$  let 
$$\theta_i=\Big|2\pi -\lim_{r\to\infty} \frac{4\pi b_i^{(r)}}{r}\Big|$$
and let $\boldsymbol\theta=(\theta_1,\dots,\theta_{|E|}).$
Then as $r$ varies over all odd integers and at $q=e^{\frac{2\pi i}{r}},$ 
$$\lim_{r\to \infty} \frac{2\pi}{r} \log \mathrm{TV}_r(M,E, \mathbf b^{(r)}) =\mathrm {Vol}(M_{E_{\boldsymbol\theta}}),$$
where $M_{E_{\boldsymbol\theta}}$ is $M$ with the hyperbolic polyhedral metric on $(M,\mathcal T)$ with cone angles $\boldsymbol\theta.$
\end{conjecture}

We note that if $\mathbf b=(0,\dots,0),$ then Conjecture \ref{VC} recovers the Volume Conjecture for the Turaev-Viro invariants for hyperbolic $3$-manifolds with totally geodesic boundary proposed by Chen and the author\,\cite{CY}. 

The main result of this paper is the following

\begin{theorem}\label{main} Conjecture \ref{VC} is true for all ideally triangulated $3$-manifold $(M,\mathcal T)$ with non-empty boundary with  sufficiently small cone angles $\boldsymbol\theta.$
\end{theorem}

In \cite[Proposition 6.14]{LY}, Luo and the author proved that, parametrized by the cone angles, hyperbolic polyhedral metrics on $(M, \mathcal T )$ with all the cone angles less than $\pi$ form a non-empty convex open polytope. In particular, when the cone angles are sufficiently small, there exists a unique hyperbolic polyhedral metric with these prescribed cone angles; and any hyperbolic polyhedral metric with cone angles less than  or equal to $\pi$ can be smoothly deformed in the space of hyperbolic polyhedral metrics to the hyperbolic polyhedral metric with all the cone angles equal to $0.$ It is expected that the space of all hyperbolic polyhedral metrics on $(M,\mathcal T)$ is connected so that each hyperbolic polyhedral metric can be smoothly deformed to the hyperbolic polyhedral metric with all the cone angles equal to $0.$ In \cite{Ko}, Kojima proved that every hyperbolic $3$-manifold $M$ with totally geodesic boundary admits an ideal triangulation such that each tetrahedron is either isometric to a truncated hyperideal tetrahedron or flat; and it is expected that every such $M$ admits a geometric ideal triangulation that each tetrahedron is truncated hyperideal. Therefore, for Kojima's ideal triangulations, if one could push the cone angles in Theorem \ref{main} from sufficiently small to $2\pi,$ then one solves Chen-Yang's Volume Conjecture\,\cite{CY} for the Turaev-Viro invariants for hyperbolic $3$-manifolds with totally geodesic boundary.

As an immediate consequence of Theorems \ref{=} and \ref{main}, we have

\begin{theorem}\label{main1.4} The Volume Conjecture of the relative Reshetikhin-Turaev invariants  \cite[Conjecture 1.1]{WY2} is true for all pairs $(D(M),D(E))$ with sufficiently small cone angles.
\end{theorem}
\medskip

\noindent\textbf{Outline of the proof of Theorem \ref{main}.}  
We follow the guideline of the method pioneered by Ohtsuki and his collaborators\,\cite{O,OY,O3,O2}. In Proposition \ref{computation}, we compute the relative Turaev-Viro invariant of $(M,\mathcal T),$ writing them as a sum of values of a holomorphic function $f_r$ at integer points. The function $f_r$ comes from Faddeev's quantum dilogarithm function. Using Poisson Summation Formula, we in Proposition \ref{Poisson} write the invariants as a sum of the Fourier coefficients of $f_r$ computed  in Propositions \ref{4.2}. In Proposition \ref{crit} we show that the critical value of the functions in the leading Fourier coefficients has real part the volume of the manifold in the hyperbolic polyhedral metric. Then we estimate the leading Fourier coefficients in Sections \ref{leading} using the Saddle Point Method (Proposition \ref{saddle}). Finally, we estimate the non-leading Fourier coefficients and the error term respectively in Sections \ref{ot} and \ref{ee} showing that they are neglectable, and prove Theorem \ref{main} in Section \ref{pf}. 
\\

\noindent\textbf{Acknowledgments.} The author would like to thank Giulio Belletti, Francis Bonahon, Qingtao Chen, Xingshan Cui, Feng Luo and Ka Ho Wong for inspiring discussions.  The author is supported by NSF Grants DMS-1812008 and DMS-2203334.
\section{Relationship with the relative Reshetikhin-Turaev invariants}\label{TVRT}

We first recall the definition of the relative Reshetikhin-Turaev invariants following the skein theoretical approach\,\cite{BHMV, Li}, and focus on the $SO(3)$-theory and the values at the root of unity  $q=e^{\frac{2\pi\sqrt{-1}}{r}}$ for odd integers $r\geqslant 3.$

A framed link in an oriented $3$-manifold $M$ is a smooth embedding $L$ of a disjoint union of finitely many thickened circles $\mathrm S^1\times [0,\epsilon],$ for some $\epsilon>0,$ into $M.$ The Kauffman bracket skein module $\mathrm K_r(M)$ of $M$ is the $\mathbb C$-module generated by the isotopic classes of framed links in $M$  modulo the follow two relations: 

\begin{enumerate}[(1)]
\item  \emph{Kauffman Bracket Skein Relation:} \ $\cp{\includegraphics[width=1cm]{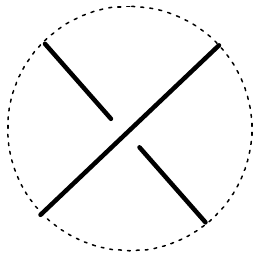}}\ =\ e^{\frac{\pi\sqrt{-1}}{r}}\ \cp{\includegraphics[width=1cm]{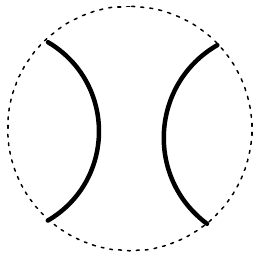}}\  +\ e^{-\frac{\pi\sqrt{-1}}{r}}\ \cp{\includegraphics[width=1cm]{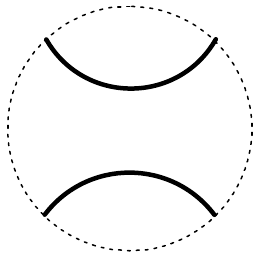}}.$ 

\item \emph{Framing Relation:} \ $L \cup \cp{\includegraphics[width=0.8cm]{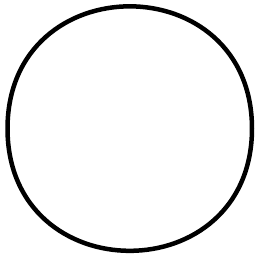}}=(-e^{\frac{2\pi\sqrt{-1}}{r}}-e^{-\frac{2\pi\sqrt{-1}}{r}})\ L.$ 
\end{enumerate}

There is a canonical isomorphism 
$$\langle\ \rangle:\mathrm K_r(\mathrm S^3)\to\mathbb C$$
defined by sending the empty link to $1.$ The image $\langle L\rangle$ of the framed link $L$ is called the Kauffman bracket of $L.$

Let $\mathrm K_r(A\times [0,1])$ be the Kauffman bracket skein module of the product of an annulus $A$ with a closed interval. For any link diagram $D$ in $\mathbb R^2$ with $k$ ordered components and $S_1, \dots, S_k\in \mathrm K_r(A\times [0,1]),$ let 
$$\langle S_1,\dots, S_k\rangle_D$$
be the complex number obtained by cabling $S_1,\dots, S_k$ along the components of $D$ considered as a element of $K_r(\mathrm S^3)$ then taking the Kauffman bracket $\langle\ \rangle.$

On $\mathrm K_r(A\times [0,1])$ there is a commutative multiplication induced by the juxtaposition of annuli, making it a $\mathbb C$-algebra; and as a $\mathbb C$-algebra $\mathrm K_r(A\times [0,1])  \cong \mathbb C[z],$ where $z$ is the core curve of the annulus  $A.$ For an integer $a\geqslant 0,$ let $e_a(z)$ be the $a$-th Chebyshev polynomial defined recursively by
$e_0(z)=1,$ $e_1(z)=z$ and $e_a(z)=ze_{a-1}(z)-e_{a-2}(z).$ Let 
$$\mathrm{I}_r=\{0,2,\dots,r-3\}$$
 be the set of even integers in between $0$ and $r-2.$ Then the Kirby coloring $\Omega_r\in\mathrm K_r(A\times [0,1])$ is defined by 
$$\Omega_r=\mu_r\sum_{a\in \mathrm{I}_r}[a+1]e_{a},$$
where
 $$\mu_r=\frac{2\sin\frac{2\pi}{r}}{\sqrt r}$$ 
 and $[a]$ is the quantum integer defined by
$$[a]=\frac{e^{\frac{2a\pi\sqrt{-1}}{r}}-e^{-\frac{2a\pi\sqrt{-1}}{r}}}{e^{\frac{2\pi\sqrt{-1}}{r}}-e^{-\frac{2\pi\sqrt{-1}}{r}}}.$$

Let $M$ be a closed oriented $3$-manifold and let $L$ be a framed link in $M$ with $n$ components. Suppose $M$ is obtained from $S^3$ by doing a surgery along a framed link $L',$ $D_{L'}$ is a standard diagram of $L'$ (ie, the blackboard framing of $D_{L'}$ coincides with the framing of $L'$). Then $L$ adds extra components to $D_{L'}$ forming a linking diagram $D_{L\cup L'}$ with $D_L$ and $D_{L'}$ linking in possibly a complicated way. Let
$U_+$ be the diagram of the unknot with framing $1,$ $\sigma(L')$ be the signature of the linking matrix of $L'$ and $\mathbf b=(b_1,\dots,b_n)$ be a multi-element of $I_r.$ Then the $r$-th \emph{relative Reshetikhin-Turaev invariant of $M$ with $L$ colored by $\mathbf b$} is defined as
\begin{equation}\label{RT}
\mathrm{RT}_r(M,L,\mathbf b)=\mu_r \langle e_{b_1},\dots,e_{b_n}, \Omega_r, \dots, \Omega_r\rangle_{D_{L\cup L'}}\langle \Omega_r\rangle _{U_+}^{-\sigma(L')}.
\end{equation}

\begin{proof}[Sketch of the proof of Theorem \ref{=}] We focus on the case that $q=e^{\frac{2\pi i}{r}},$ and the case that  $q=e^{\frac{\pi i}{r}}$ is similar.

Consider the handle decomposition of $M$ dual to the ideal triangulation $\mathcal T,$ namely, the $2$-handles come from a tubular neighborhood of the edges, the $1$-handles come from a tubular neighborhood of the farces and the $0$-handles come from the complement of the $1$- and $2$-handles. Following the idea of Roberts \cite{R}, we construct the following quantity $CM_r(M,E,\mathbf b).$ Let $\{\epsilon_1,\dots, \epsilon_{|E|}\}$ be the attaching curves of the $2$-handles and let $\{\delta_1,\dots, \delta_{|F|}\}$ be the meridians of the $1$-handles. Thicken these curves to bands parallel to the surface of the $1$-skeleton $H$ and push each $\epsilon_i$ slightly into $H$ and circulate it by a framed trivial loop $\gamma_i.$ Embed $H$ arbitrarily into $S^3,$  cable each of the image of the $\epsilon$- and $\delta$-bands by the Kirby coloring $\Omega_r$ and cable the image of $\gamma_i$ by the $b_i$-th Chebyshev polynomial. In this way, we get an element $S_{(M,E,\mathbf b)}$ in $K_r(S^3),$ and we define
$$CM_r(M,E,\mathbf b)=\mu_r^{|T|}\langle S_{(M,E,\mathbf b)}\rangle.$$

On the one hand, since each face of $\mathcal T$ has three edges, each $\delta$-band encloses exactly three $\epsilon$-bands (see \cite[Figure 11]{R}).
Writing 
$$\Omega_r=\mu_r\sum_{a\in \mathrm{I}_r}[a+1]e_{a}$$
and applying \cite[Lemma 3.3]{DKY} to each $\delta$-band, we have
\begin{equation}\label{CM=TV}
\begin{split}
CM_r(M,E,\mathbf b)=&\mu_r^{|T|-|E|+|F|}\sum_{\mathbf a}\prod_{i=1}^{|E|} \mathrm H(a_i, b_i)\prod_{s=1}^{|T|}\bigg|\begin{matrix} a_{s_1} & a_{s_2}  & a_{s_3} \\  a_{s_4}  &  a_{s_5}  &  a_{s_6}  \end{matrix} \bigg|,\\
=&2^{-\mathrm{rank H}(M;\mathbb Z_2)}\mu_r^{\chi(M)}\mathrm{TV_r}(M,E,\mathbf b),
\end{split}
\end{equation}
where $\mathbf c$ runs over all the $r$-admissible colorings of $(M,\mathcal T)$ with even integers and the last equality comes from \cite[Lemma A.4, Theorem 2.9 and its proof]{DKY}.

On the other hand, the image of the union of the $\epsilon$- and $\delta$-bands form a surgery diagram of the $3$-manifold
$D(M)\#(S^2\times S^1)^{\# (|T|-1)}$ and the signature of the linking matrix  equals zero as argued in \cite[Proof of Theorem 3.4]{R}. Note that  the solid tori attached along the $\epsilon$-bands are the double of the $2$-handles, and each $\gamma_i$ is isotopic to the meridian of a tubular neighborhood of $\epsilon_i,$ hence is isotopic to the core of the solid torus attached to it, which is the double of the edge $e_i$ of $\mathcal T.$ Therefore,
\begin{equation}\label{CM=RT}
\begin{split}
CM_r(M,E,\mathbf b)=&\mu_r^{|T|-1}\mathrm{RT_r}(D(M)\#(S^2\times S^1)^{\# (|T|-1)},D(E),\mathbf b)\\
=&\mathrm{RT_r}(D(M),D(E),\mathbf b),
\end{split}
\end{equation}
where the last equality comes from the fact that 
$$\mathrm{RT}_r(M_1\#M_2,L_1\cup L_2, (\mathbf b_1,\mathbf b_2))=\mu_r^{-1}\mathrm{RT}_r(M_1,L_1,\mathbf b_1)\cdot\mathrm{RT}_r(M_2,L_2,\mathbf b_2)$$
 and that $\mathrm{RT}_r(S^2\times S^1)=1.$

From (\ref{CM=TV}) and (\ref{CM=RT}), the result follows.
\end{proof}


\section{Preliminaries}

\subsection{Hyperbolic polyhedral metrics}

A \emph{hyperbolic polyhedral metric}\,\cite{Luo, LY} on an ideally triangulated $3$-manifold $(M,\mathcal T)$ is obtained by replacing each tetrahedron in $\mathcal T$ by a truncated hyperideal tetrahedron (see Section \ref{CCS}) and replacing the gluing homeomorphisms between pairs of the faces  by isometries. The \emph{cone angle} at an edge is the sum of the dihedral angles of the truncated hyperideal tetrahedra around the edge. If all the cone angles are equal to $2\pi,$ then the hyperbolic polyhedral metric gives a hyperbolic metric on $M$ with totally geodesic boundary. 

 In \cite[Theorem 1.2 (b)]{LY}, Luo and the author proved that hyperbolic polyhedral metrics on $(M, \mathcal T )$ are rigid in the sense that they are up to isometry determined by their cone angles; and in \cite[Proposition 6.14]{LY}, they proved that hyperbolic polyhedral metrics on $(M, \mathcal T )$ with all cone angles less than $\pi$ form a non-empty convex open polytope when parametrized by the cone angles. In particular, when the cone angles are sufficiently small, there exists a unique hyperbolic polyhedral metric with these prescribed cone angles; and any hyperbolic polyhedral metric with cone angles less than or equal to $\pi$ can be smoothly deformed in the space of hyperbolic polyhedral metrics to the hyperbolic polyhedral metric with all the cone angles equal to $0.$ It is expected that the space of all  hyperbolic polyhedral metrics on $(M,\mathcal T)$ is connected so that each hyperbolic polyhedral metric can be smoothly deformed to the hyperbolic polyhedral metric with all the cone angles equal to $0.$

\subsection{Truncated hyperideal tetrahedra}\label{CCS}

We first recall some of the basic results on hyperideal tetrahedra. Following \cite{BB} and \cite{Fu}, a truncated
\emph{hyperideal tetrahedron} $\Delta$ in $\mathbb{H}^3$ is a
compact convex polyhedron that is diffeomorphic to a truncated
tetrahedron in $\mathbb{E}^3$ with four hexagonal faces $\{H_1,H_2,H_3,H_4\}$ isometric to
right-angled hyperbolic
 hexagons and four triangular faces $\{T_1,T_2,T_3,T_4\}$ isometric to hyperbolic triangles (see Figure \ref{hyperideal}). 
\begin{figure}[htbp]
\centering
\includegraphics[scale=0.3]{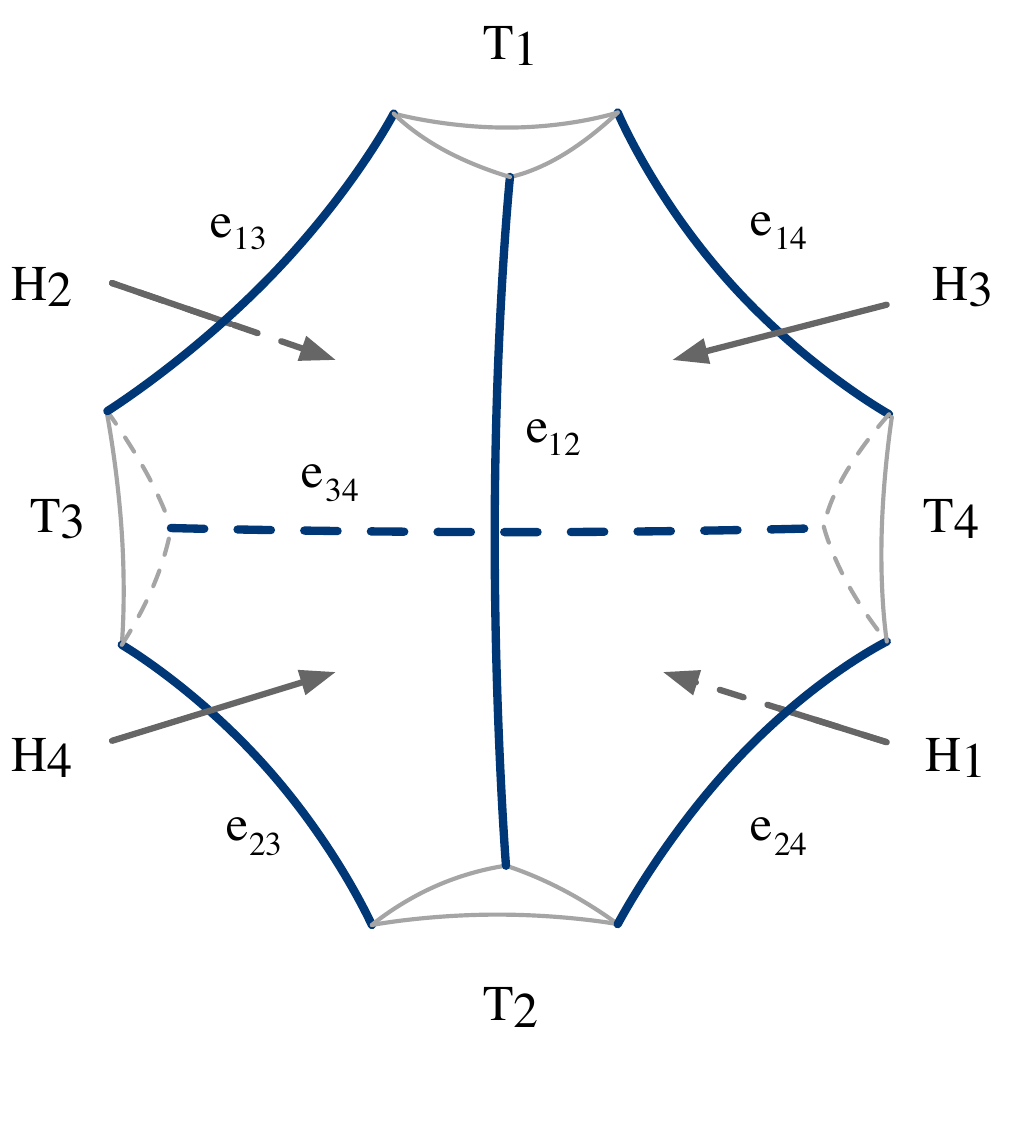}
\caption{}
\label{hyperideal} 
\end{figure}
An \emph{edge} in a hyperideal tetrahedron is the
 intersection of two hexagonal faces and the \emph{dihedral angle} at an edge is the angle between the two hexagonal faces
   adjacent to it. The angle between a hexagonal face and
   a triangular face is always $\frac{\pi}{2}.$ Let $e_{ij}$ be the edge connecting the triangular faces $T_i$ and $T_j,$ and let $\theta_{ij}$ and $l_{ij}$ respectively be the dihedral angle at and edge length of $e_{ij}.$ Then by the Cosine Law of hyperbolic triangles and right-angled hyperbolic hexagons, we have
\begin{equation}\label{cos1}
\cosh\l_{ij}=\frac{c_{kl}+c_{ik}c_{jk}+c_{il}c_{jl}+(c_{ik}c_{jl}+c_{il}c_{jk})c_{ij}-c_{kl}c_{ij}^2}{\sqrt{-1+c_{ij}^2+c_{ik}^2+c_{il}^2+2c_{ij}c_{ik}c_{il}}\sqrt{-1+c_{ij}^2+c_{jk}^2+c_{jl}^2+2c_{ij}c_{jk}c_{jl}}},
\end{equation}
where $\{i,j,k,l\}=\{1,2,3,4\}$ and $c_{ij}=\cos\theta_{ij};$
and
\begin{equation}\label{cos2}
\cos\theta_{kl}=\frac{ch_{ij}+ch_{ik}ch_{il}+ch_{jk}ch_{jl}+(ch_{ik}ch_{jl}+ch_{il}ch_{jk})ch_{ij}-ch_{kl}ch_{ij}^2}{\sqrt{-1+ch_{ij}^2+ch_{ik}^2+ch_{jk}^2+2ch_{ij}ch_{ik}ch_{jk}}\sqrt{-1+ch_{ij}^2+ch_{il}^2+ch_{jl}^2+2ch_{ij}ch_{il}ch_{jl}}},
\end{equation}
where $ch_{ij}=\cosh\l_{ij}.$

By \cite{BB}, a truncated hyperideal tetrahedron is up to isometry determined by its six dihedral angles $\{\theta_1,\dots, \theta_6\},$ and by (\ref{cos2}) is determined by its six edge lengths $\{l_1,\dots,l_6\}.$ 

\begin{definition}[\cite{Luo, LY}]\label{cov}
Let $\mathrm{Vol}$ and $\{\theta_1,\dots,\theta_6\}$ respectively be the volume  and the dihedral angles of a hyperideal tetrahedron $\Delta$ as functions of the edge lengths $\{l_1,\dots,l_6\}.$ The co-volume function $\mathrm{Cov}$  is defined by
$$\mathrm{Cov}(l_1,\dots,l_6)=\mathrm{Vol}+\frac{1}{2}\sum_{i=1}^6\theta_i\cdot l_i .$$
\end{definition}

The key property of the co-volume function is the following

\begin{lemma}\label{Sch} For $i\in \{1,\dots,6\},$
$$\frac{\partial \mathrm{Cov}}{\partial l_i}=\frac{\theta_i}{2}.$$
\end{lemma}

\begin{proof} By the Schl\"afli formula, we have
$$\frac{\partial \mathrm{Vol}}{\partial \theta_i}=-\frac{l_i}{2}.$$
Then by the chain rule and the product rue, we have
\begin{equation*}
\begin{split}
\frac{\partial \mathrm{Cov}}{\partial l_i}=&\sum_{k=1}^6\frac{\partial \mathrm{Vol}}{\partial\theta_k}\cdot\frac{\partial \theta_k}{\partial l_i}+\frac{1}{2}\sum_{k=1}^6\frac{\partial}{\partial l_i}\Big(\theta_k\cdot l_k\Big)\\
=&-\sum_{k=1}^6\frac{l_k}{2}\cdot \frac{\partial \theta_k}{\partial l_i}+\frac{1}{2}\sum_{k=1}^6\cdot\frac{\partial \theta_k}{\partial l_i}\cdot l_k+\frac{\theta_i}{2}=\frac{\theta_i}{2}.
\end{split}
\end{equation*}
\end{proof}


\subsection{Dilogarithm and quantum dilogarithm functions}\label{6jsymbol}

Let $\log:\mathbb C\setminus (-\infty, 0]\to\mathbb C$ be the standard logarithm function defined by
$$\log z=\log|z|+\sqrt{-1}\arg z$$
with $-\pi<\arg z<\pi.$
 
The dilogarithm function $\mathrm{Li}_2: \mathbb C\setminus (1,\infty)\to\mathbb C$ is defined by
$$\mathrm{Li}_2(z)=-\int_0^z\frac{\log (1-u)}{u}du$$
where the integral is along any path in $\mathbb C\setminus (1,\infty)$ connecting $0$ and $z,$ which is holomorphic in $\mathbb C\setminus [1,\infty)$ and continuous in $\mathbb C\setminus (1,\infty).$

By eg. Zagier\,\cite{Z}), on the unit circle $\big\{ z=e^{2\sqrt{-1}\theta}\,\big|\,0 \leqslant \theta\leqslant\pi\big\},$ 
\begin{equation}\label{dilogLob}
\mathrm{Li}_2(e^{2\sqrt{-1}\theta})=\frac{\pi^2}{6}+\theta(\theta-\pi)+2\sqrt{-1}\Lambda(\theta),
\end{equation}
where  $\Lambda:\mathbb R\to\mathbb R$  is the Lobachevsky function defined by
$$\Lambda(\theta)=-\int_0^\theta\log|2\sin t|dt,$$
which is an odd function of period $\pi.$ See eg. Thurston's notes\,\cite[Chapter 7]{T}.


The following variant of Faddeev's quantum dilogarithm functions\,\cite{F, FKV} will play a key role in the proof of the main result. 
Let $r\geqslant 3$ be an odd integer. Then the following contour integral
\begin{equation}
\varphi_r(z)=\frac{4\pi\sqrt{-1}}{r}\int_{\Omega}\frac{e^{(2z-\pi)x}}{4x \sinh (\pi x)\sinh (\frac{2\pi x}{r})}\ dx
\end{equation}
defines a holomorphic function on the domain $$\Big\{z\in \mathbb C \ \Big|\ -\frac{\pi}{r}<\mathrm{Re}z <\pi+\frac{\pi}{r}\Big\},$$  
  where the contour is
$$\Omega=\big(-\infty, -\epsilon\big]\cup \big\{z\in \mathbb C\ \big||z|=\epsilon, \mathrm{Im}z>0\big\}\cup \big[\epsilon,\infty\big),$$
for some $\epsilon\in(0,1).$
Note that the integrand has poles at $n\sqrt{-1},$ $n\in\mathbb Z,$ and the choice of  $\Omega$ is to avoid the pole at $0.$
\\

The function $\varphi_r(z)$ satisfies the following fundamental properties, whose proof can be found in \cite[Section 2.3]{WY}.

\begin{lemma}\label{factorial} Let $\{n\}=q^n-q^{-n}$ and $\{n\}!=\prod_{k=1}^n\{k\}.$ 
\begin{enumerate}[(1)]
\item For $0\leqslant n \leqslant r-2,$
\begin{equation}
\{n\}!=e^{\frac{r}{4\pi\sqrt{-1}}\Big(-2\pi\big(\frac{2\pi n}{r}\big)+\big(\frac{2\pi}{r}\big)^2(n^2+n)+\varphi_r\big(\frac{\pi}{r}\big)-\varphi_r\big(\frac{2\pi n}{r}+\frac{\pi}{r}\big)\Big)}.
\end{equation}
\item For $\frac{r-1}{2}\leqslant n \leqslant r-2,$
\begin{equation} \label{move}
\{n\}!=2e^{\frac{r}{4\pi\sqrt{-1}}\Big(-2\pi\big(\frac{2\pi n}{r}\big)+\big(\frac{2\pi }{r}\big)^2(n^2+n)+\varphi_r\big(\frac{\pi}{r}\big)-\varphi_r\big(\frac{2\pi n}{r}+\frac{\pi}{r}-\pi\big)\Big)}.
\end{equation}
\end{enumerate}
\end{lemma}

We consider (\ref{move}) because there are poles in $(\pi,2\pi),$ and to avoid the poles we move the variables to $(0,\pi)$ by subtracting $\pi.$

The function $\varphi_r(z)$ and the dilogarithm function are closely related as follows.

\begin{lemma}\label{converge}  \begin{enumerate}[(1)]
\item For every $z$ with $0<\mathrm{Re}z<\pi,$ 
\begin{equation}\label{conv}
\varphi_r(z)=\mathrm{Li}_2(e^{2\sqrt{-1}z})+\frac{2\pi^2e^{2\sqrt{-1}z}}{3(1-e^{2\sqrt{-1}z})}\frac{1}{r^2}+O\Big(\frac{1}{r^4}\Big).
\end{equation}
\item For every $z$ with $0<\mathrm{Re}z<\pi,$ 
\begin{equation}\label{conv}
\varphi_r'(z)=-2\sqrt{-1}\log(1-e^{2\sqrt{-1}z})+O\Big(\frac{1}{r^2}\Big).
\end{equation}
\item \cite[Formula (8)(9)]{O2}
$$\varphi_r\Big(\frac{\pi}{r}\Big)=\mathrm{Li}_2(1)+\frac{2\pi\sqrt{-1}}{r}\log\Big(\frac{r}{2}\Big)-\frac{\pi^2}{r}+O\Big(\frac{1}{r^2}\Big).$$
\end{enumerate}\end{lemma}


\subsection{Quantum $6j$-symbols and their underlying geometry} 

As is customary we define
 $$[a]!=\prod_{k=1}^a[k].$$

Recall that a triple $(a_i,a_j,a_k)$ of integers in $\{0,\dots,r-2\}$ is \emph{$r$-admissible} if 
\begin{enumerate}[(1)]
\item  $a_i+a_j-a_k\geqslant 0,$
\item $a_i+a_j+a_k\leqslant 2(r-2),$ and
\item $a_i+a_j+a_k$ is even.
\end{enumerate}

For an $r$-admissible triple $(a_1,a_2,a_3)$, define 
$$\Delta(a_1,a_2,a_3)=\sqrt{\frac{[\frac{a_1+a_2-a_3}{2}]![\frac{a_2+a_3-a_1}{2}]![\frac{a_3+a_1-a_2}{2}]!}{[\frac{a_1+a_2+a_3}{2}+1]!}}$$
with the convention that $\sqrt{x}=\sqrt{|x|}\sqrt{-1}$ when the real number $x$ is negative.

A  6-tuple $(a_1,\dots,a_6)$ is \emph{$r$-admissible} if the triples $(a_1,a_2,a_3)$, $(a_1,a_5,a_6)$, $(a_2,a_4,a_6)$ and $(a_3,a_4,a_5)$ are $r$-admissible

\begin{definition}
The \emph{quantum $6j$-symbol} of an $r$-admissible 6-tuple $(a_1,\dots,a_6)$ is 
\begin{multline*}
\bigg|\begin{matrix} a_1 & a_2 & a_3 \\ a_4 & a_5 & a_6 \end{matrix} \bigg|
= \sqrt{-1}^{-\sum_{i=1}^6a_i}\Delta(a_1,a_2,a_3)\Delta(a_1,a_5,a_6)\Delta(a_2,a_4,a_6)\Delta(a_3,a_4,a_5)\\
\sum_{k=\max \{T_1, T_2, T_3, T_4\}}^{\min\{ Q_1,Q_2,Q_3\}}\frac{(-1)^k[k+1]!}{[k-T_1]![k-T_2]![k-T_3]![k-T_4]![Q_1-k]![Q_2-k]![Q_3-k]!},
\end{multline*}
where $T_1=\frac{a_1+a_2+a_3}{2}$, $T_2=\frac{a_1+a_5+a_6}{2}$, $T_3=\frac{a_2+a_4+a_6}{2}$ and $T_4=\frac{a_3+a_4+a_5}{2}$, $Q_1=\frac{a_1+a_2+a_4+a_5}{2}$, $Q_2=\frac{a_1+a_3+a_4+a_6}{2}$ and $Q_3=\frac{a_2+a_3+a_5+a_6}{2}.$
\end{definition}

Closely related, a triple $(\alpha_1,\alpha_2,\alpha_3)\in [0,2\pi]^3$ is \emph{admissible} if 
\begin{enumerate}[(1)]
\item $\alpha_i+\alpha_j-\alpha_k\geqslant 0$ for $\{i,j,k\}=\{1,2,3\},$ and
\item $\alpha_i+\alpha_j+\alpha_k\leqslant 4\pi.$
\end{enumerate}
A $6$-tuple $(\alpha_1,\dots,\alpha_6)\in [0,2\pi]^6$ is \emph{admissible} if the triples $\{1,2,3\},$ $\{1,5,6\},$ $\{2,4,6\}$ and $\{3,4,5\}$ are admissible.

\begin{definition} An $r$-admissible $6$-tuple $(a_1,\dots,a_6)$  is of the \emph{hyperideal type} if for $\{i,j,k\}=\{1,2,3\},$ $\{1,5,6\},$ $\{2,4,6\}$ and $\{3,4,5\},$
\begin{enumerate}[(1)]
\item $0\leqslant a_i+a_j-a_k<r-2,$
\item $r-2< a_i+a_j+a_k\leqslant 2(r-2),$ and
\item $a_i+a_j+a_k$ is even.
\end{enumerate}
\end{definition}

As a consequence of Lemma \ref{factorial} we have

\begin{proposition}\label{6jqd} The quantum $6j$-symbol at the root of unity $q=e^{\frac{2\pi \sqrt{-1}}{r}}$ can be computed as 
$$\bigg|
\begin{matrix}
        a_1 & a_2 & a_3 \\
        a_4 & a_5 & a_6 
      \end{matrix} \bigg|=\frac{\{1\}}{2}\sum_{k=\max\{T_1,T_2,T_3,T_4\}}^{\min\{Q_1,Q_2,Q_3,r-2\}}e^{\frac{r}{4\pi \sqrt{-1}}U_r\big(\frac{2\pi a_1}{r},\dots,\frac{2\pi a_6}{r},\frac{2\pi k}{r}\big)},$$
 where $U_r$ is defined as follows. If $(a_1,\dots,a_6)$ is of the hyperideal type, then
\begin{equation}\label{termwithr}
\begin{split}
U_r(\alpha_1,\dots,\alpha_6,\xi)=&\pi^2-\Big(\frac{2\pi}{r}\Big)^2+\frac{1}{2}\sum_{i=1}^4\sum_{j=1}^3(\eta_j-\tau_i)^2-\frac{1}{2}\sum_{i=1}^4\Big(\tau_i+\frac{2\pi}{r}-\pi\Big)^2\\
&+\Big(\xi+\frac{2\pi}{r}-\pi\Big)^2-\sum_{i=1}^4(\xi-\tau_i)^2-\sum_{j=1}^3(\eta_j-\xi)^2\\
&-2\varphi_r\Big(\frac{\pi}{r}\Big)-\frac{1}{2}\sum_{i=1}^4\sum_{j=1}^3\varphi_r\Big(\eta_j-\tau_i+\frac{\pi}{r}\Big)+\frac{1}{2}\sum_{i=1}^4\varphi_r\Big(\tau_i-\pi+\frac{3\pi}{r}\Big)\\
&-\varphi_r\Big(\xi-\pi+\frac{3\pi}{r}\Big)+\sum_{i=1}^4\varphi_r\Big(\xi-\tau_i+\frac{\pi}{r}\Big)+\sum_{j=1}^3\varphi_r\Big(\eta_j-\xi+\frac{\pi}{r}\Big),\\
\end{split}
\end{equation}
where $\alpha_i=\frac{2\pi a_i}{r}$ for $i=1,\dots,6$ and $\xi=\frac{2\pi k}{r},$ $\tau_1=\frac{\alpha_1+\alpha_2+\alpha_3}{2}$, $\tau_2=\frac{\alpha_1+\alpha_5+\alpha_6}{2}$, $\tau_3=\frac{\alpha_2+\alpha_4+\alpha_6}{2}$ and $\tau_4=\frac{\alpha_3+\alpha_4+\alpha_5}{2}$, $\eta_1=\frac{\alpha_1+\alpha_2+\alpha_4+\alpha_5}{2}$, $\eta_2=\frac{\alpha_1+\alpha_3+\alpha_4+\alpha_6}{2}$ and $\eta_3=\frac{\alpha_2+\alpha_3+\alpha_5+\alpha_6}{2}.$
If $(a_1,\dots,a_6)$ is not of the hyperideal type, then $U_r$ will be changed according to Lemma \ref{factorial}.
\end{proposition}

\begin{definition} A $6$-tuple $(\alpha_1,\dots,\alpha_6)\in [0,2\pi]^6$ is of the \emph{hyperideal type} if
\begin{enumerate}[(1)]
\item $0\leqslant \alpha_i+\alpha_j-\alpha_k\leqslant 2\pi,$ and
\item $2\pi\leqslant \alpha_i+\alpha_j+\alpha_k\leqslant 4\pi.$
\end{enumerate}
\end{definition}

We notice that the six numbers $|\pi-\alpha_1|,\dots,|\pi-\alpha_6|$ are the dihedral angles of an ideal or a hyperideal tetrahedron if and only if $(\alpha_1,\dots,\alpha_6)$ is of the hyperideal type. 

By Lemma \ref{converge}, $U_r=U-\frac{4\pi\sqrt{-1}}{r}\log\big(\frac{r}{2}\big)+O(\frac{1}{r}),$ where  $U$ is defined by
\begin{equation}\label{term}
\begin{split}
U(\alpha_1,\dots,\alpha_6,\xi)=&\pi^2+\frac{1}{2}\sum_{i=1}^4\sum_{j=1}^3(\eta_j-\tau_i)^2-\frac{1}{2}\sum_{i=1}^4(\tau_i-\pi)^2\\
&+(\xi-\pi)^2-\sum_{i=1}^4(\xi-\tau_i)^2-\sum_{j=1}^3(\eta_j-\xi)^2\\
&-2\mathit{Li}_2(1)-\frac{1}{2}\sum_{i=1}^4\sum_{j=1}^3\mathit{Li}_2\big(e^{2i(\eta_j-\tau_i)}\big)+\frac{1}{2}\sum_{i=1}^4\mathit{Li}_2\big(e^{2i(\tau_i-\pi)}\big)\\
&-\mathit{Li}_2\big(e^{2i(\xi-\pi)}\big)+\sum_{i=1}^4\mathit{Li}_2\big(e^{2i(\xi-\tau_i)}\big)+\sum_{j=1}^3\mathit{Li}_2\big(e^{2i(\eta_j-\xi)}\big)\\
\end{split}
\end{equation}
on the region $\mathrm{B_{H,\mathbb C}}$ consisting of $(\alpha_1,\dots,\alpha_6,\xi)\in\mathbb C^7$ such that $(\mathrm{Re}(\alpha_1),\dots,\mathrm{Re}(\alpha_6))$ is of the hyperideal type and $\max\{\mathrm{Re}(\tau_i)\}\leqslant \mathrm{Re}(\xi)\leqslant \min\{\mathrm{Re}(\eta_j), 2\pi\}.$ 

Let 
$$\mathrm{B_H}=\mathrm{B_{H,\mathbb C}}\cap \mathbb R^7.$$

For $\boldsymbol\alpha=(\alpha_1,\dots,\alpha_6)\in\mathbb C^6$ such that $(\mathrm{Re}(\alpha_1),\dots,\mathrm{Re}(\alpha_6))$ is of the hyperideal type, we let $\xi(\boldsymbol\alpha)$ be such that
\begin{equation}\label{xia}
\frac{\partial U(\boldsymbol\alpha,\xi)}{\partial \xi}\Big|_{\xi=\xi(\boldsymbol\alpha)}=0.
\end{equation}
Following the idea of \cite{MY}, see also \cite{BY}, it is proved that $e^{-2\sqrt{-1}\xi(\boldsymbol\alpha)}$ satisfies a concrete quadratic equation. Therefore, for each such $\boldsymbol\alpha,$ there is at most one $\xi(\boldsymbol\alpha)$ such that $(\boldsymbol\alpha,\xi(\boldsymbol\alpha))\in \mathrm{B_{H,\mathbb C}}.$  Indeed, only one of the two solutions of the quadratic equation makes (\ref{xia}) hold, and the other only makes it hold modulo a multiple of $\pi\sqrt{-1}.$ In \cite[Section 3]{BY}, it shows that  if all $\mathrm{Re}(\alpha_1),\dots,\mathrm{Re}(\alpha_6)$ are sufficiently close to $\pi,$ then $(\boldsymbol\alpha,\xi(\boldsymbol\alpha))\in \mathrm{B_{H,\mathbb C}}.$

For $\boldsymbol\alpha\in\mathbb C^6$ so that $(\boldsymbol\alpha,\xi(\boldsymbol\alpha))\in\mathrm{B_{H,\mathbb C}},$ we define
\begin{equation}\label{W}
W(\boldsymbol\alpha)=U(\boldsymbol\alpha,\xi(\boldsymbol\alpha)).
\end{equation}
 Then as a special case of \cite[Theorem 3.5]{BY},  we have
 
 \begin{theorem}\label{co-vol} 
 $$W\big(\pi\pm \sqrt{-1} l_1,\dots, \pi\pm \sqrt{-1} l_6\big)=2\pi^2+2\sqrt{-1}\cdot\mathrm{Cov}\big(l_1,\dots, l_6\big)$$
for all $\{l_1,\dots,l_6\}$ that form the set of edge lengths of a truncated hyperideal tetrahedron, where $\mathrm{Cov}$ is the co-volume function defined in Definition \ref{cov}.
\end{theorem}


\section{Computation of the relative Turaev-Viro invariants}

\begin{proposition}\label{computation} Let $\mathbf b$ be a coloring of $(M,\mathcal T).$ Then the relative Turaev-Viro invariant $\mathrm{TV}_r(M,E,\mathbf b)$ of $(M,\mathcal T)$ at the root of unity $q=e^{\frac{2\pi \sqrt{-1}}{r}}$ can be computed as 
$$\mathrm {TV}_r(M,E,\mathbf b)=\frac{(-1)^{|E|\big(\frac{r}{2}+1\big)}2^{\mathrm{rank H}(M;\mathbb Z_2)-|T|}}{\{1\}^{|E|-|T|}}\sum_{\mathbf a,\mathbf k}\Big(\sum_{\epsilon}g_r^{\epsilon}(\mathbf a,\mathbf k)\Big),$$
where $\epsilon=(\epsilon_1,\dots,\epsilon_{|E|})\in\{1,-1\}^E$ runs over all multi-signs, $\mathbf a=(a_1,\dots,a_{|E|})$ runs over all multi-even integers in $\{0,2,\dots,r-3\}$ so that for each $s\in\{1,\dots,|T|\}$ the $6$-tuple $(a_{s_1},a_{s_2},a_{s_3}, a_{s_4},a_{s_5},a_{s_6})$  is  $r$-admissible, and $\mathbf k=(k_1,\dots,k_{|T|})$ runs over all multi-integers with each $k_s$ lying in between $\max\{T_{s_i}\}$ and $\min\{Q_{s_j},r-2\},$ with
$$g_r^{\epsilon}(\mathbf a,\mathbf k)=e^{\sum_{i=1}^{|E|}\epsilon_i\frac{2\pi\sqrt{-1}(a_i+b_i+1)}{r}+\frac{r}{4\pi\sqrt{-1}}\mathcal W_r^{\epsilon}(\frac{2\pi \mathbf a}{r},\frac{2\pi\mathbf k}{r})}$$
where $\frac{2\pi \mathbf a}{r}=\Big(\frac{2\pi a_1}{r},\dots,\frac{2\pi a_{|E|}}{r}\Big),$ $\frac{2\pi \mathbf k}{r}=\Big(\frac{2\pi k_1}{r},\dots,\frac{2\pi k_{|T|}}{r}\Big),$ and
\begin{equation*}
\begin{split}
\mathcal W_r^{\epsilon}(\boldsymbol\alpha,\boldsymbol\xi)=-\sum_{i=1}^{|E|}2\epsilon_i(\alpha_i-\pi)(\beta_i-\pi)+\sum_{s=1}^{|T|} U_r(\alpha_{s_1},\dots,\alpha_{s_6},\xi_s)
\end{split}
\end{equation*}
with $\alpha_i=\frac{2\pi a_i}{r}$ and $\beta_i=\frac{2\pi b_i}{r}$ for $i=1,\dots,|E|,$ $\boldsymbol\alpha=(\alpha_1,\dots,\alpha_{|E|})$ and $\boldsymbol\xi=(\xi_1,\dots,\xi_{|T|}).$
\end{proposition}

\begin{proof} First, we observe that if we let the summation in the definition of $\mathrm{TV}_r(M,E,\mathbf b)$ be over all multi-even integers $\mathbf a$ instead of multi-integers, then the resulting quantity differs from $\mathrm{TV}_r(M,E,\mathbf b)$ by a factor $2^{\mathrm{rank H}(M;\mathbb Z_2)}$ by \cite[Lemma A.4, Theorem 2.9 and its proof]{DKY}.  Next, we observe that
$$(-1)^{a+b}q^{(a+1)(b+1)}=-(-1)^{\frac{r}{2}}q^{\big(a-\frac{r}{2}\big)\big(b-\frac{r}{2}\big)+a+b+1},$$
and
$$(-1)^{a+b}q^{-(a+1)(b+1)}=(-1)^{\frac{r}{2}}q^{-\big(a-\frac{r}{2}\big)\big(b-\frac{r}{2}\big)-a-b-1}.$$
As a consequence, we have
\begin{equation*}
\begin{split}
\mathrm H(a_i,b_i)=&\frac{1}{q-q^{-1}}\bigg((-1)^{a_i+b_i}q^{(a_i+1)(b_i+1)}-(-1)^{a_i+b_i}q^{-(a_i+1)(b_i+1)}\bigg)\\
=&\frac{-(-1)^{\frac{r}{2}}}{q-q^{-1}}\bigg(q^{\big((a_i-\frac{r}{2})(b_i-\frac{r}{2})+a_i+b_i+1\big)}+q^{-\big((a_i-\frac{r}{2})(b_i-\frac{r}{2})+a_i+b_i+1\big)}\bigg)\\
=&\frac{(-1)^{\frac{r}{2}+1}}{\{1\}}\sum_{\epsilon_i\in\{-1,1\}}q^{\epsilon_i\big((a_i-\frac{r}{2})(b_i-\frac{r}{2})+a_i+b_i+1\big)}\\
=&\frac{(-1)^{\frac{r}{2}+1}}{\{1\}}\sum_{\epsilon_i\in\{-1,1\}}e^{\epsilon_i\frac{2\sqrt{-1}\pi(a_i+b_i+1)}{r}+\frac{r}{4\pi\sqrt{-1}}\Big(-2\epsilon_i\big(\frac{2\pi a_i}{r}-\pi\big)\big(\frac{2\pi b_i}{r}-\pi\big)\Big)},\\
\end{split}
\end{equation*}
and hence
\begin{equation*}
\begin{split}\prod_{i\in I}\mathrm H(a_i,b_i)=&\frac{(-1)^{|I|\big(\frac{r}{2}+1\big)}}{\{1\}^{|I|}}\sum_{\epsilon_I\in\{-1,1\}^{|I|}}e^{\sum_{i=1}^{|E|}\epsilon_i\frac{2\sqrt{-1}\pi(a_i+b_i+1)}{r}+\frac{r}{4\pi\sqrt{-1}}\sum_{i=1}^{|E|}\Big(-2\epsilon_i\big(\frac{2\pi a_i}{r}-\pi\big)\big(\frac{2\pi b_i}{r}-\pi\big)\Big)}\\
=&\frac{(-1)^{|I|\big(\frac{r}{2}+1\big)}}{\{1\}^{|I|}}\sum_{\epsilon_I\in\{-1,1\}^{|I|}}e^{\sum_{i=1}^{|E|}\epsilon_i\sqrt{-1}(\alpha_i+\beta_i+\frac{2\pi}{r})+\frac{r}{4\pi\sqrt{-1}}\sum_{i=1}^{|E|}\big(-2\epsilon_i(\alpha_i-\pi)(\beta_i-\pi)\big)}.\\
\end{split}
\end{equation*}

Then the result follows from Proposition \ref{6jqd}.
\end{proof}

We notice that the summation in Proposition \ref{computation} is finite, and to use the Poisson Summation Formula, we need an infinite sum over integral points. To this end, we consider the following regions and a bump function over them. 

Let $\alpha_i=\frac{2\pi a_i}{r}$ and $\beta_i=\frac{2\pi b_i}{r}$ for $i=1,\dots,|E|,$ $\xi_s=\frac{2\pi k_s}{r}$ for $s=1,\dots, |T|,$ $\tau_{s_i}=\frac{2\pi T_{s_i}}{r}$ for $i=1,\dots,4,$ and $\eta_{s_j}=\frac{2\pi Q_{s_j}}{r}$ for $j=1,2,3.$ Let
$$\mathrm {D_A}=\Big\{(\boldsymbol\alpha,\boldsymbol\xi)\in\mathbb R^{|E|+|T|}\ \Big|\ (\alpha_{s_1},\dots,\alpha_{s_6}) \text{ is admissible, } \max\{\tau_{s_i}\}\leqslant \xi_s\leqslant \min\{\eta_{s_j}, 2\pi\}, s=1,\dots, |T|\Big\},$$
and let
$$\mathrm {D_H}=\Big\{(\boldsymbol\alpha,\boldsymbol\xi)\in\mathrm {D_A} \ \Big|\ (\alpha_{s_1},\dots,\alpha_{s_6}) \text{ is of the hyperideal type}, s=1,\dots, |T| \Big\}.$$
For a sufficiently small $\delta >0,$ let 
$$\mathrm {D_H^\delta}=\Big\{(\boldsymbol\alpha,\boldsymbol\xi)\in\mathrm {D_H}\ \Big|\ d((\boldsymbol\alpha,\boldsymbol\xi), \partial\mathrm {D_H})>\delta \Big\},$$
where $d$ is the Euclidean distance on $\mathbb R^n.$
We let $\psi:\mathbb R^{|E|+|T|}\to\mathbb R$ be the $C^{\infty}$-smooth bump function supported on $(\mathrm{D_H}, \mathrm{D_H^\delta}),$ ie,  \begin{equation*}
\left \{\begin{array}{rl}
\psi(\boldsymbol\alpha,\boldsymbol\xi)=1, & (\boldsymbol\alpha,\boldsymbol\xi)\in \overline{\mathrm{D_H^\delta}}\\
0<\psi(\boldsymbol\alpha,\boldsymbol\xi)<1, &  (\boldsymbol\alpha,\boldsymbol\xi)\in \mathrm{D_H}\setminus \overline{\mathrm{D_H^\delta}}\\
\psi(\boldsymbol\alpha,\boldsymbol\xi)=0, & (\boldsymbol\alpha,\boldsymbol\xi)\notin \mathrm{D_H},\\
\end{array}\right.
\end{equation*}
and let 
$$f^{\epsilon}_r(\mathbf a,\mathbf k)=\psi\Big(\frac{2\pi \mathbf a}{r},\frac{2\pi \mathbf k}{r}\Big)g^{\epsilon}_r(\mathbf a,\mathbf k).$$

In Proposition \ref{computation}, the coloring $\mathbf a$ runs over multi-even integers. On the other hand, to use the Poisson Summation Formula, we need a sum over all integers. For this purpose, we for each $i$ let $a_i=2a_i'$ and let $\mathbf a'=(a_1',\dots,a_{|E|}').$ 
Then by Proposition \ref{computation}, 
$$\mathrm {TV}_r(M,E,\mathbf b)=\frac{(-1)^{|E|\big(\frac{r}{2}+1\big)}2^{\mathrm{rank H}(M;\mathbb Z_2)-|T|}}{\{1\}^{|E|-|T|}}\sum_{(\mathbf a',\mathbf k)\in\mathbb Z^{|E|+|T|}}\Big(\sum_{\epsilon\in\{1,-1\}^E} f_r^{\epsilon}\big(2\mathbf a',\mathbf k\big)\Big)+\text{error term}.$$
Let $$f_r=\sum_{\epsilon\in\{1,-1\}^E} f_r^{\epsilon}.$$ Then
$$\mathrm {TV}_r(M,E,\mathbf b)=\frac{(-1)^{|E|\big(\frac{r}{2}+1\big)}2^{\mathrm{rank H}(M;\mathbb Z_2)-|T|}}{\{1\}^{|E|-|T|}}\sum_{(\mathbf a',\mathbf k)\in\mathbb Z^{|E|+|T|}} f_r\big(2\mathbf a',\mathbf k\big)+\text{error term}.$$

Since $f_r$ is $C^{\infty}$-smooth and equals zero out of $\mathrm{D_H},$ it is in the Schwartz space on $\mathbb R^{|E|+|T|}.$ Then by the Poisson Summation Formula (see e.g. \cite[Theorem 3.1]{SS}),
$$\sum_{(\mathbf a',\mathbf k)\in\mathbb Z^{|E|+|T|}}f_r\big(2\mathbf a',\mathbf k\big)=\sum_{(\mathbf m,\mathbf n)\in\mathbb Z^{|E|+|T|}}\widehat {f_r}(\mathbf m,\mathbf n),$$
where $\mathbf m=(m_1,\dots,m_{|E|})\in \mathbb Z^E,$ $\mathbf n=(n_1,\dots, n_{|T|})\in \mathbb Z^T,$ and  $\widehat f_r(\mathbf m,\mathbf n)$ is the $(\mathbf m,\mathbf n)$-th Fourier coefficient of $f_r$ defined by
\begin{equation*}
\begin{split}
\widehat {f_r}(\mathbf m,\mathbf n)=\int_{\mathbb R^{|E|+|T|}}&f_r\big(2\mathbf a',\mathbf k\big)e^{\sum_{i=1}^{|E|}2\pi \sqrt{-1}m_ia_i'+\sum_{s=1}^{|T|}2\pi \sqrt{-1}n_sk_s}d\mathbf a'd\mathbf k,
\end{split}
\end{equation*}
where $d\mathbf a'=\prod_{i=1}^{|E|}da'_i$ and $d\mathbf k=\prod_{s=1}^{|T|}dk_s.$

By the change of variable, and by changing $2a_i'$ back to $a_i,$ the Fourier coefficients can be computed as
\begin{proposition}\label{4.2}
$$\widehat{f_r}(\mathbf m,\mathbf n)=\sum_{\epsilon\in\{1,-1\}^E} \widehat{f^{\epsilon}_r}(\mathbf m,\mathbf n)$$
with
\begin{equation*}
\begin{split}
\widehat{f^{\epsilon}_r}(\mathbf m,\mathbf n)=\frac{ r^{|E|+|T|}}{2^{2|E|+|T|}\cdot\pi^{|E|+|T|}}&\int_{\mathrm{D_H}}\psi(\boldsymbol\alpha,\boldsymbol\xi)e^{\sum_{i=1}^{|E|}\epsilon_i\sqrt{-1}(\alpha_i+\beta_i+\frac{2\pi}{r})}\\ 
&\cdot e^{\frac{r}{4\pi \sqrt{-1}}\big(\mathcal W_r^{\epsilon}(\boldsymbol\alpha,\boldsymbol\xi)-\sum_{i=1}^{|E|}2\pi m_i\alpha_i-\sum_{s=1}^{|E|}4\pi n_s\xi_s\big)}d\boldsymbol\alpha d\boldsymbol\xi,
\end{split}
\end{equation*}
where $d\boldsymbol\alpha=\prod_{i=1}^{|E|}d\alpha_i,$ $d\boldsymbol\xi=\prod_{s=1}^{|T|}d\xi_s$ and
$$\mathcal W_r^{\epsilon}(\boldsymbol\alpha,\boldsymbol\xi)=-\sum_{i=1}^{|E|}2\epsilon_i(\alpha_i-\pi)(\beta_i-\pi)+\sum_{s=1}^{|T|}U_r(,\alpha_{s_1},\dots,\alpha_{s_6},\xi_s).$$
In particular,
\begin{equation*}
\begin{split}
\widehat{f^{\epsilon}_r}(\mathbf 0,\mathbf 0)=\frac{r^{|E|+|T|}}{2^{2|E|+|T|}\cdot\pi^{|E|+|T|}}\int_{\mathrm{D_H}}\psi(\boldsymbol\alpha,\boldsymbol\xi)e^{\sum_{i=1}^{|E|}\epsilon_i\sqrt{-1}(\alpha_i+\beta_i+\frac{2\pi}{r})+\frac{r}{4\pi \sqrt{-1}}\mathcal W_r^{\epsilon}(\boldsymbol\alpha,\boldsymbol\xi)}d\boldsymbol\alpha d\boldsymbol\xi.
\end{split}
\end{equation*}
\end{proposition}

\begin{proposition}\label{Poisson}
$$\mathrm {TV}_r(M,E,\mathbf b)=\frac{(-1)^{|E|\big(\frac{r}{2}+1\big)}2^{\mathrm{rank H}(M;\mathbb Z_2)-|T|}}{\{1\}^{|E|-|T|}}\sum_{(\mathbf m,\mathbf n)\in\mathbb Z^{|E|+|T|}}\widehat{ f_r}(\mathbf m,\mathbf n)+\text{error term}.$$
\end{proposition}

We will estimate the leading Fourier coefficients, the non-leading Fourier coefficients and the error term respectively in Sections \ref{leading}, \ref{ot} and \ref{ee}, and prove Theorem \ref{main} in Section \ref{pf}.


\section{Asymptotics}\label{Asy}

The goal of this section is to prove Theorem \ref{main}. The main tool we use is Proposition \ref{saddle}, which is a generalization of the standard Saddle Point Approximation \cite{O} and whose proof can be found in \cite[Appendix]{WY2}. With a geometric preparation in Sections \ref{cv} and \ref{convex}, we  in Section \ref{leading} estimate the leading Fourier coefficients, and respectively in Sections \ref{ot} and \ref{ee} estimate the non-leading Fourier coefficients and the error term in the formula of the relative Turaev-Viro invariants obtained in  Propositions \ref{4.2} and \ref{Poisson}.

\begin{proposition}\label{saddle}
Let $D_{\mathbf z}$ be a region in $\mathbb C^n$ and let $D_{\mathbf a}$ be a region in $\mathbb R^k.$ Let $f(\mathbf z,\mathbf a)$ and $g(\mathbf z,\mathbf a)$ be complex valued functions on $D_{\mathbf z}\times D_{\mathbf a}$  which are holomorphic in $\mathbf z$ and smooth in $\mathbf a.$ For each positive integer $r,$ let $f_r(\mathbf z,\mathbf a)$ be a complex valued function on $D_{\mathbf z}\times D_{\mathbf a}$ holomorphic in $\mathbf z$ and smooth in $\mathbf a.$
For a fixed $\mathbf a\in D_{\mathbf a},$ let $f^{\mathbf a},$ $g^{\mathbf a}$ and $f_r^{\mathbf a}$ be the holomorphic functions  on $D_{\mathbf z}$ defined by
$f^{\mathbf a}(\mathbf z)=f(\mathbf z,\mathbf a),$ $g^{\mathbf a}(\mathbf z)=g(\mathbf z,\mathbf a)$ and $f_r^{\mathbf a}(\mathbf z)=f_r(\mathbf z,\mathbf a).$ Suppose $\{\mathbf a_r\}$ is a convergent sequence in $D_{\mathbf a}$ with $\lim_r\mathbf a_r=\mathbf a_0,$ $f_r^{\mathbf a_r}$ is of the form
$$ f_r^{\mathbf a_r}(\mathbf z) = f^{\mathbf a_r}(\mathbf z) + \frac{\upsilon_r(\mathbf z,\mathbf a_r)}{r^2},$$
$\{S_r\}$ is a sequence of embedded real $n$-dimensional closed disks in $D_{\mathbf z}$ sharing the same boundary and converging to an embedded $n$-dimensional closed disk $S_0,$ and $\mathbf c_r$ is a point on $S_r$ such that  $\{\mathbf c_r\}$ is convergent in $D_{\mathbf z}$ with $\lim_r\mathbf c_r=\mathbf c_0.$ If for each $r$
\begin{enumerate}[(1)]
\item $\mathbf c_r$ is a critical point of $f^{\mathbf a_r}$ in $D_{\mathbf z},$
\item $\mathrm{Re}f^{\mathbf a_r}(\mathbf c_r) > \mathrm{Re}f^{\mathbf a_r}(\mathbf z)$ for all $\mathbf z \in S\setminus \{\mathbf c_r\},$
\item the domain $\{\mathbf z\in D_{\mathbf z}\ |\ \mathrm{Re}f^{\mathbf a_r}(\mathbf z)<f^{\mathbf a_r}(\mathbf c_r)\}$ deformation retracts to $S_r\setminus\{\mathbf c_r\},$
\item $|g^{\mathbf a_r}(\mathbf c_r)|$ is bounded from below by a positive constant independent of $r,$
\item $|\upsilon_r(\mathbf z, \mathbf a_r)|$ is bounded from above by a constant independent of $r$ on $D_{\mathbf z},$ and
\item  the Hessian matrix $\mathrm{Hess}(f^{\mathbf a_0})$ of $f^{\mathbf a_0}$ at $\mathbf c_0$ is non-singular,
\end{enumerate}
then
\begin{equation*}
\begin{split}
 \int_{S_r} g^{\mathbf a_r}(\mathbf z) e^{rf_r^{\mathbf a_r}(\mathbf z)} d\mathbf z= \Big(\frac{2\pi}{r}\Big)^{\frac{n}{2}}\frac{g^{\mathbf a_r}(\mathbf c_r)}{\sqrt{-\det\mathrm{Hess}(f^{\mathbf a_r})(\mathbf c_r)}} e^{rf^{\mathbf a_r}(\mathbf c_r)} \Big( 1 + O \Big( \frac{1}{r} \Big) \Big).
 \end{split}
 \end{equation*}
\end{proposition}

For a fixed $\{\beta_1,\dots,\beta_{|E|}\},$ let $\theta_i=2|\beta_i-\pi|$ for each $i\in\{1,\dots, |E|\}.$  The function $\mathcal W_r^{\epsilon}$ is approximated by the following function
$$\mathcal W^{\epsilon}(\boldsymbol\alpha,\boldsymbol\xi)=-\sum_{i=1}^{|E|}2\epsilon_i(\alpha_i-\pi)(\beta_i-\pi)+\sum_{s=1}^{|T|}U(\alpha_{s_1},\dots,\alpha_{s_6},\xi_s).$$
The approximation will be specified in the proof of Proposition \ref{critical}. Notice that $\mathcal W^{\epsilon}$ is continuous on 
$$\mathrm{D_{H,\mathbb C}}=\big\{(\boldsymbol\alpha,\boldsymbol\xi)\in\mathbb C^{|E|+|T|}\ \big|\ (\mathrm{Re}(\boldsymbol\alpha),\mathrm{Re}(\boldsymbol\xi))\in \mathrm{D_{H}}\big\}$$ and for any $\delta>0$ is analytic on 
$$\mathrm{D^\delta_{H,\mathbb C}}=\big\{(\boldsymbol\alpha,\boldsymbol\xi)\in\mathbb C^{|E|+|T|}\ \big|\ (\mathrm{Re}(\boldsymbol\alpha),\mathrm{Re}(\boldsymbol\xi))\in \mathrm{D^\delta_{H}}\big\},$$ 
where $\mathrm{Re}(\boldsymbol\alpha)=(\mathrm{Re}(\alpha_1),\dots, \mathrm{Re}(\alpha_{|E|}))$ and $\mathrm{Re}(\boldsymbol\xi)=(\mathrm{Re}(\xi_1),\dots, \mathrm{Re}(\xi_{|T|})).$

In the rest of this paper, we assume that 	$\theta_1,\dots,\theta_{|E|}$ are sufficiently close to $0,$ or equivalently, $\beta_1,\dots,\beta_{|E|}$ are sufficiently close to $\pi.$ In the special case $\beta_i=\dots=\beta_{|E|}=\pi,$ a direct computation shows that $\xi(\pi,\dots,\pi)=\frac{7\pi}{4}.$ For $\delta>0,$ we denote by $\mathrm{D_{\delta,\mathbb C}}$  the $L^\infty$ $\delta$-neighborhood  of $\big(\pi,\dots,\pi,\frac{7\pi}{4},\dots,\frac{7\pi}{4}\big)$ in $\mathbb C^{|E|+|T|},$  that is 
$$\mathrm{D_{\delta,\mathbb C}}=\Big\{(\boldsymbol\alpha,\boldsymbol\xi)\in \mathbb C^{|E|+|T|}\ \Big|\ d_{L^\infty}\Big((\boldsymbol\alpha,\boldsymbol\xi),\Big(\pi,\dots,\pi,\frac{7\pi}{4},\dots,\frac{7\pi}{4}\Big)\Big)<\delta\Big\},$$
where $d_{L^\infty}$ is the real $L^\infty$ norm on $\mathbb C^n$ defined by
$$d_{L^\infty}(\mathbf x,\mathbf y)=\max_{i\in\{1,\dots,n\}}\{|\mathrm {Re}(x_i)-\mathrm{Re}(y_i)|, |\mathrm {Im}(x_i)-\mathrm{Im}(y_i)| \},$$
where $\mathbf x=(x_1,\dots,x_n)$ and $\mathbf y=(y_1,\dots,y_n).$  We will also consider the region 
$$\mathrm{D_{\delta}}=\mathrm{D_{\delta,\mathbb C}}\cap \mathbb R^{|E|+|T|}.$$

\subsection{Critical points and critical values of $\mathcal W^{\epsilon}$}\label{cv}

Suppose $\{\beta_1,\dots,\beta_{|E|}\}$ are sufficiently close to $\pi.$ Let $\theta_i=2|\beta_i-\pi|$ for $i\in \{1,\dots,|E|\},$  and let $\mu_i=1$ if $\beta_i\geqslant\pi$ and let $\mu_i=-1$ if $\beta_i\leqslant \pi$ so that 
$$\mu_i\theta_i=2(\beta_i-\pi).$$

\begin{proposition}\label{crit} For each $i\in \{1,\dots,|E|\},$ let $l_i$ be the length of the edge $e_i$ in $M_{E_{\boldsymbol\theta}}$ (the hyperbolic polyhedral metric on $(M,\mathcal T)$ with cone angles $\boldsymbol\theta),$ and let
\begin{equation}\label{alpha}
\alpha^*_i=\pi+\epsilon_i\mu_i\sqrt{-1}l_i.
\end{equation}
For each $s\in\{1,\dots,|T|\},$  let $\xi^*_s=\xi(\alpha^*_{s_1},\dots,\alpha^*_{s_6}).$
Then $\mathcal W^{\epsilon}$ has a critical point 
$$\boldsymbol z^{\epsilon}=\big(\alpha^*_1,\dots,\alpha^*_{|E|},\xi^*_1,\dots,\xi^*_{|T|}\big)$$
in $\mathrm{D_{\delta,\mathbb C}}$ with critical value $$2|T|\pi^2+2\sqrt{-1}\mathrm{Vol}(M_{E_{\boldsymbol\theta}}).$$
\end{proposition}

\begin{proof}   
For each $s\in\{1,\dots,|T|\},$  let $\boldsymbol\alpha_s=(\alpha_{s_1},\dots,\alpha_{s_6})$ and let $\boldsymbol\alpha^*_s=(\alpha^*_{s_1},\dots,\alpha^*_{s_6}).$ 

By (\ref{cos1}), if $\theta_i$'s are sufficiently small, then $l_i$'s are sufficiently close to $0$ and $\alpha^*_i$'s are sufficiently close to $\pi.$ Then by the continuity of $\xi(\boldsymbol\alpha_s)$ for each $s,$ $\boldsymbol z^{\epsilon}\in \mathrm D_{\delta,\mathbb C}.$

We first have
\begin{equation}\label{c1}
\frac{\partial \mathcal W^{\epsilon}}{\partial \xi_s}\Big|_{\boldsymbol z^\epsilon}=\frac{\partial U(\boldsymbol \alpha^*_s,\xi_s)}{\partial \xi_s}\Big|_{\xi^*_s}=0.
\end{equation}

Now let $W(\boldsymbol\alpha_s)=U(\boldsymbol\alpha_s,\xi(\boldsymbol\alpha_s))$ be the function defined in (\ref{W}). Then for $i\in \{s_1,\dots,s_6\},$
$$\frac{\partial  W(\boldsymbol\alpha_s)}{\partial \alpha_i}\Big|_{\boldsymbol\alpha^*_s}=\frac{\partial U(\boldsymbol\alpha_s,\xi_s)}{\partial \alpha_i}\Big|_{(\boldsymbol\alpha^*_s,\xi^*_s)}+\frac{\partial U(\boldsymbol\alpha_s,\xi_s)}{\partial \xi_s}\Big|_{(\boldsymbol\alpha^*_s,\xi^*_s)}\cdot\frac{\partial \xi(\boldsymbol\alpha_s)}{\partial \alpha_i}\Big|_{\boldsymbol\alpha^*_s}=\frac{\partial U(\boldsymbol\alpha_s,\xi_s)}{\partial \alpha_i}\Big|_{(\boldsymbol\alpha^*_s,\xi^*_s)}.$$
For each $s\in\{1,\dots,|T|\},$ let $(l_{s_1},\dots,l_{s_6})$ be the edge lengths  of $\Delta_s.$ Then by Theorem \ref{co-vol} and Lemma \ref{Sch}, we have 
\begin{equation*}
\begin{split}
\frac{\partial U(\boldsymbol\alpha_s,\xi_s)}{\partial \alpha_i}\Big|_{(\boldsymbol\alpha^*_s,\xi^*_s)}=\frac{\partial  W(\boldsymbol\alpha_s)}{\partial \alpha_i}\Big|_{\boldsymbol\alpha^*_s}=-\epsilon_i\mu_i \sqrt{-1}\cdot\frac{\partial W}{\partial l_i}\Big|_{(l_{s_1},\dots,l_{s_6})}=\epsilon_i\mu_i \theta_{s,i},
\end{split}
\end{equation*}
where $\theta_{s,i}$ is the dihedral angle of $\Delta_s$ at the edge $e_i.$ 
Then for each $i\in\{1,\dots, |E|\},$ 
\begin{equation}\label{c2}
\begin{split}
\frac{\partial \mathcal W^{\epsilon}}{\partial \alpha_i}\Big|_{\boldsymbol z^{\epsilon}}=&-2\epsilon_i(\beta_i-\pi)+\sum_{s=1}^{|T|}\frac{\partial U(\boldsymbol\alpha_s,\xi_s)}{\partial \alpha_i}\Big|_{(\boldsymbol\alpha^*_s,\xi^*_s)}\\
=&-2\epsilon_i(\beta_i-\pi)+\epsilon_i\mu_i\sum_{s=1}^{|E|} \theta_{s,i}=\epsilon_i\big(-2(\beta_i-\pi)+\mu_i\theta_i\big)=0.
\end{split}
\end{equation}
By (\ref{c1}) and (\ref{c2}), $\boldsymbol z^\epsilon$ is a critical point of $\mathcal W^\epsilon.$

Finally, we compute the critical value. For each $s\in\{1,\dots,|T|\},$ let $(l_{s_1},\dots,l_{s_6})$ and $(\theta_{s_1},\dots,\theta_{s_6})$ respectively be the edge lengths and the dihedral angles of $\Delta_s.$ Then by Theorem \ref{co-vol}, we have 
\begin{equation*}
\begin{split}
\mathcal W^{\epsilon}(\boldsymbol z^{\epsilon})=&-\sum_{i=1}^{|E|}2\epsilon_i(\sqrt{-1}\epsilon_i\mu_il_i)(\beta_i-\pi)+\sum_{s=1}^{|T|}\Big(2\pi^2+2\sqrt{-1}\Big(\mathrm{Vol}(\Delta_s)+\frac{1}{2}\sum_{k=1}^6\theta_{s_k}l_{s_k}\Big)\Big)\\
=&2|T|\pi^2+2\sum_{s=1}^{|T|}\sqrt{-1}\mathrm{Vol}(\Delta_s)+\sum_{i=1}^{|E|}2\sqrt{-1}\Big(-\mu_i(\beta_i-\pi)+\sum_{s=1}^{|T|}\theta_{s,i}\Big)l_i\\
=&2|T|\pi^2+2\sqrt{-1}\mathrm{Vol}(M_{E_\theta}).
\end{split}
\end{equation*}
\end{proof}


\subsection{ Convexity of $\mathcal W^{\epsilon}$}\label{convex}

\begin{proposition}\label{convexity} 
For a sufficiently small $\delta_0>0,$ the function $\mathcal W^{\epsilon}(\boldsymbol\alpha,\boldsymbol\xi)$ is strictly concave down in $\{\mathrm{Re}(\alpha_i)\}_{i=1}^{|E|} $ and $\{\mathrm{Re}(\xi_s)\}_{s=1}^{|T|},$ and is strictly concave up in $\{\mathrm{Im}(\alpha_i)\}_{i=1}^{|E|}$ and $\{\mathrm{Im}(\xi_s)\}_{s=1}^{|T|}$  on $\mathrm{D_{\delta_0,\mathbb C}}.$
\end{proposition}

\begin{proof} We first consider the special case $\{\alpha_i\}_{i=1}^{|E|}$ and $\{\xi_s\}_{s=1}^{|T|}$ are real.  In this case, 
$$\mathrm{Im}\mathcal W^{\epsilon}(\boldsymbol\alpha,\boldsymbol\xi)=\sum_{s=1}^{|T|}2V(\alpha_{s_1},\dots,\alpha_{s_6},\xi_s)$$
for $V:\mathrm{B_H}\to\mathbb R$ defined by 
\begin{equation}\label{V}
\begin{split}
V(\alpha_1,\dots,\alpha_6,\xi)=\,&\delta(\alpha_1,\alpha_2,\alpha_3)+\delta(\alpha_1,\alpha_5,\alpha_6)+\delta(\alpha_2,\alpha_4,\alpha_6)+\delta(\alpha_3,\alpha_4,\alpha_5)\\
&-\Lambda(\xi)+\sum_{i=1}^4\Lambda(\xi-\tau_i)+\sum_{j=1}^3\Lambda(\eta_j-\xi),
\end{split}
\end{equation}
where $\delta$ is defined by
$$\delta(\alpha,\beta,\gamma)=-\frac{1}{2}\Lambda\Big(\frac{\alpha+\beta-\gamma}{2}\Big)-\frac{1}{2}\Lambda\Big(\frac{\beta+\gamma-\alpha}{2}\Big)-\frac{1}{2}\Lambda\Big(\frac{\gamma+\alpha-\beta}{2}\Big)+\frac{1}{2}\Lambda\Big(\frac{\alpha+\beta+\beta}{2}\Big).$$

At $\big(\pi,\dots,\pi,\frac{7\pi}{4}\big),$ we have $\frac{\partial ^2\mathrm{Im}V}{\partial \alpha_{s_i}^2} =-2$ for $i\in \{1,\dots,6\},$ $\frac{\partial ^2V}{\partial \alpha_{s_i}\alpha_{s_j}} =-1$ for $i\neq j$ in $\{1,\dots,6\},$ $\frac{\partial ^2\mathrm{Im}V}{\partial \alpha_{s_i}\xi_s} =2$ for $i\in \{1,\dots,6\}$ and  $\frac{\partial ^2\mathrm{Im}V}{\partial \xi_s^2} =-8.$ Then a direct computation shows that,  at $\big(\pi,\dots,\pi,\frac{7\pi}{4}\big),$ the Hessian matrix 
of $\mathrm{Im}V$ in $\{\mathrm{Re}(\alpha_{s_i})\}_{i\in \{1,\dots,6\}}$ and $\mathrm{Re}(\xi_s)$ is negative definite. As a consequence, the Hessian matrix 
of $\mathrm{Im}\mathcal W^{\epsilon}$ in $\{\mathrm{Re}(\alpha_i)\}_{i\in I}$ and $\{\mathrm{Re}(\xi_s)\}_{s=1}^c$ is negative definite at $\big(\pi,\dots,\pi,\frac{7\pi}{4},\dots,\frac{7\pi}{4} \big).$

Then by the continuity, there exists a sufficiently small $\delta_0>0$ such that $(\boldsymbol\alpha,\boldsymbol\xi)\in \mathrm{D_{\delta_0,\mathbb C}},$ the Hessian matrix of $\mathrm{Im}\mathcal W^{\epsilon}$ with respect to $\{\mathrm{Re}(\alpha_i)\}_{i=1}^{|E|}$ and $\{\mathrm{Re}(\xi_s)\}_{s=1}^{|T|}$ is still negative definite, implying that $\mathrm{Im}\mathcal W^{\epsilon}$ is strictly concave down in $\{\mathrm{Re}(\alpha_i)\}_{i=1}^{|E|}$ and  $\{\mathrm{Re}(\xi_s)\}_{s=1}^{|T|}$ on $\mathrm{D_{\delta_0,\mathbb C}}.$ Since $\mathcal W^{\epsilon}$ is holomorphic, $\mathrm{Im}\mathcal W^{\epsilon}$ is strictly concave up in $\{\mathrm{Im}(\alpha_i)\}_{i=1}^{|E|}$ and $\{\mathrm{Im}(\xi_s)\}_{s=1}^{|T|}$ on $\mathrm{D_{\delta_0,\mathbb C}}.$
\end{proof}

\begin{proposition}\label{nonsingular} The Hessian matrix $\mathrm{Hess}\mathcal W^{\epsilon}$ of $\mathcal W^{\epsilon}$ with respect to $\{\alpha_i\}_{i=1}^{|E|}$ and $\{\xi_s\}_{s=1}^{|T|}$ is non-singular on $\mathrm{D_{\delta_0,\mathbb C}}.$
\end{proposition}

\begin{proof} By Proposition \ref{convexity}, the real part of the $\mathrm{Hess}\mathcal W^{\epsilon}$ is negative definite. Then by \cite[Lemma]{L}, it is nonsingular.
\end{proof}


\subsection{Asymptotics of the leading Fourier coefficients}\label{leading}

\begin{proposition}\label{critical}
Suppose $\{\beta_1,\dots,\beta_{|E|}\}$ are in $\{\pi-\epsilon,\pi+\epsilon
\}$ for a sufficiently small $\epsilon>0.$ For $\epsilon\in\{1,-1\}^E,$ let  $\boldsymbol z^{\epsilon}$ be the critical point of $\mathcal W^{\epsilon}$ described in Proposition \ref{crit}. Then
$$\widehat{f^{\epsilon}_r}(0,\dots,0)=\frac{C^{\epsilon}(\boldsymbol z^{\epsilon})}{\sqrt{-\det\mathrm{Hess}\Big(\frac{\mathcal W^{\epsilon}(z^{\epsilon})}{4\pi \sqrt{-1}}\Big)}}e^{\frac{r}{2\pi}\mathrm{Vol}(M_{E_{\boldsymbol\theta}})}\Big( 1 + O \Big( \frac{1}{r} \Big) \Big),$$
where each $C^{\epsilon}(\boldsymbol z^{\epsilon})$ depends continuously on $\{\beta_1,\dots,\beta_{|E|}\}$ and when $\beta_1=\dots=\beta_{|E|}=\pi,$
$$C^{\epsilon}(\boldsymbol z^{\epsilon})=\frac{(-1)^{|T|}r^{\frac{|E|-|T|}{2}}}{2^{\frac{3|E|+|T|}{2}}\pi^{\frac{|E|+|T|}{2}}}.$$
\end{proposition}

 For the proof of Proposition \ref{critical}, we need the following

\begin{lemma}\label{absm} For each $\epsilon\in\{1,-1\}^E,$ 
$$\max_{\mathrm{D_H}}\mathrm{Im}\mathcal W^{\epsilon} \leqslant \mathrm{Im}\mathcal W^{\epsilon}\Big(\pi,\dots,\pi,\frac{7\pi}{4},\dots,\frac{7\pi}{4}\Big)=2|T|v_8$$
where $v_8$ is the volume of the regular ideal octahedron, and the equality holds if and only if $\alpha_1=\dots=\alpha_{|E|}=\pi$ and $\xi_1=\dots=\xi_{|T|}=\frac{7\pi}{4}.$
\end{lemma}

\begin{proof} On $\mathrm{D_H},$ we have 
$$\mathrm{Im}\mathcal W^{\epsilon}(\boldsymbol \alpha,\xi_1)=\sum_{s=1}^{|T|}2V(\alpha_{s_1},\dots,\alpha_{s_6},\xi_s)$$ 
for $V$ defined in (\ref{V}). Then the result is a consequence of the result of Costantino\,\cite{C} and the Murakami-Yano formula\,\cite{MY} (see Ushijima\,\cite{U} for the case of hyperideal tetrahedra). Indeed, by \cite{C}, for a fixed $\boldsymbol \alpha=(\alpha_1,\dots,\alpha_6)$ of the hyperideal type, the function $f(\xi)$ defined by $f(\xi)=V(\boldsymbol \alpha,\xi)$ is strictly concave down and the unique maximum point $\xi(\boldsymbol \alpha)$ exists and lies in $(\max\{\tau_i\},\min\{\eta_j,2\pi\}),$ ie, $(\boldsymbol\alpha,\xi(\boldsymbol\alpha))\in\mathrm{B_H}.$ Then by \cite{U}, $V(\boldsymbol\alpha,\xi(\boldsymbol\alpha))=\mathrm{Vol}(\Delta_{|\boldsymbol\pi-\boldsymbol\alpha|}),$ the volume of the hyperideal tetrahedron $\Delta_{|\boldsymbol\pi-\boldsymbol\alpha|}$ with dihedral angles $|\pi-\alpha_1|,\dots, |\pi-\alpha_6|.$ Since $\xi(\pi,\dots,\pi)=\frac{7\pi}{4}$ and the regular ideal octahedron $\Delta_{(0,\dots,0)}$ has the maximum volume among all the hyperideal tetrahedra, $V\big(\pi,\dots,\pi,\frac{7\pi}{4}\big)=\mathrm{Vol}(\Delta_{(0,\dots,0)})\geqslant \mathrm{Vol}(\Delta_{|\boldsymbol\pi-\boldsymbol\alpha|})=V(\boldsymbol\alpha,\xi(\boldsymbol\alpha))\geqslant V(\boldsymbol\alpha,\xi)$ for any $(\boldsymbol\alpha,\xi)\in \mathrm{B_H}.$ 

For the equality part, suppose $(\alpha_1,\dots,\alpha_{|E|},\xi_1,\dots,\xi_{|T|})\neq \big(\pi,\dots,\pi,\frac{7\pi}{4},\frac{7\pi}{4}\big).$ If $(\alpha_1,\dots,\alpha_6)\neq (\pi,\dots,\pi),$ then 
$\mathrm{Im}\mathcal W^{\epsilon}(\boldsymbol\alpha,\xi)\leqslant \sum_{s=1}^{|T|}\mathrm{Vol}(\Delta_s)<2|T|v_8,$ where $\Delta_s$ is the truncated hyperideal tetrahedron with dihedral angles $|\pi-\alpha_{s_1}|,\dots,|\pi-\alpha_{s_6}|.$ If $(\alpha_{s_1},\dots,\alpha_{s_6})=(\pi,\dots,\pi)$ for all $s\in\{1,\dots, |T|\}$ but, say, $\xi_1\neq \frac{7\pi}{4},$ then the strict concavity of $f(\xi)$ implies that 
$$\mathrm{Im}\mathcal W^{\epsilon}(\pi,\dots,\pi,\xi_1,\dots,\xi_{|T|})< \mathrm{Im}\mathcal W^{\epsilon}\big(\pi,\dots,\pi,\frac{7\pi}{4},\dots,\frac{7\pi}{4}\big).$$
\end{proof}

\begin{proof}[Proof of Proposition \ref{critical}] Let $\delta_0>0$ be as in Proposition \ref{convexity}.
By Lemma \ref{convexity}, Proposition \ref{absm} and the compactness of $\mathrm{D_H}\setminus\mathrm{D_{\delta_0}},$
$$2|T|v_8>\max_{\mathrm{D_H}\setminus\mathrm{D_{\delta_0}}} \mathrm{Im}\mathcal W^{\epsilon}.$$
By Proposition \ref{crit} and continuity, if $\{\beta_1,\dots,\beta_{|E|}\}$  are sufficiently close to $\pi,$ then the critical point $z^{\epsilon}$ of $\mathcal W^{\epsilon}$  as in Proposition \ref{crit} lies in $\mathrm{D_{\delta_0,\mathbb C}},$ and $\mathrm{Im}\mathcal W^{\epsilon}(\boldsymbol z^{\epsilon})=\mathrm{Vol}(M_{E_{\boldsymbol\theta}})$ is sufficiently close to $2|T|v_8$ so that
 $$\mathrm{Im}\mathcal W^{\epsilon}(\boldsymbol z^{\epsilon})>\max_{\mathrm{D_H}\setminus\mathrm{D_{\delta_0}}} \mathrm{Im}\mathcal W^{\epsilon}.$$
 
Therefore, we only need to estimate  the integral on $\mathrm{D_{\delta_0}}.$ To do this, we consider as drawn in Figure \ref{surface} the surface $S^{\epsilon}=S^{\epsilon}_{\text{top}}\cup S^{\epsilon}_{\text{side}}$ in $\overline{\mathrm{D_{\delta_0,\mathbb C}}},$ where
$$S^{\epsilon}_{\text{top}}=\{ (\boldsymbol\alpha,\boldsymbol\xi)\in \mathrm{D_{\delta_0,\mathbb C}}\ |\ ((\mathrm{Im}(\boldsymbol\alpha)),\mathrm{Im}(\boldsymbol\xi))=\mathrm{Im}(\boldsymbol z^{\epsilon})\}$$
and
$$S^{\epsilon}_{\text{side}}=\{ (\boldsymbol\alpha,\boldsymbol\xi)+t\sqrt{-1}\cdot \mathrm{Im}(\boldsymbol z^{\epsilon})\ |\ (\boldsymbol\alpha,\boldsymbol\xi)\in\partial \mathrm{D_{\delta_0}},t\in[0,1]\}.$$

\begin{figure}[htbp]
\centering
\includegraphics[scale=0.3]{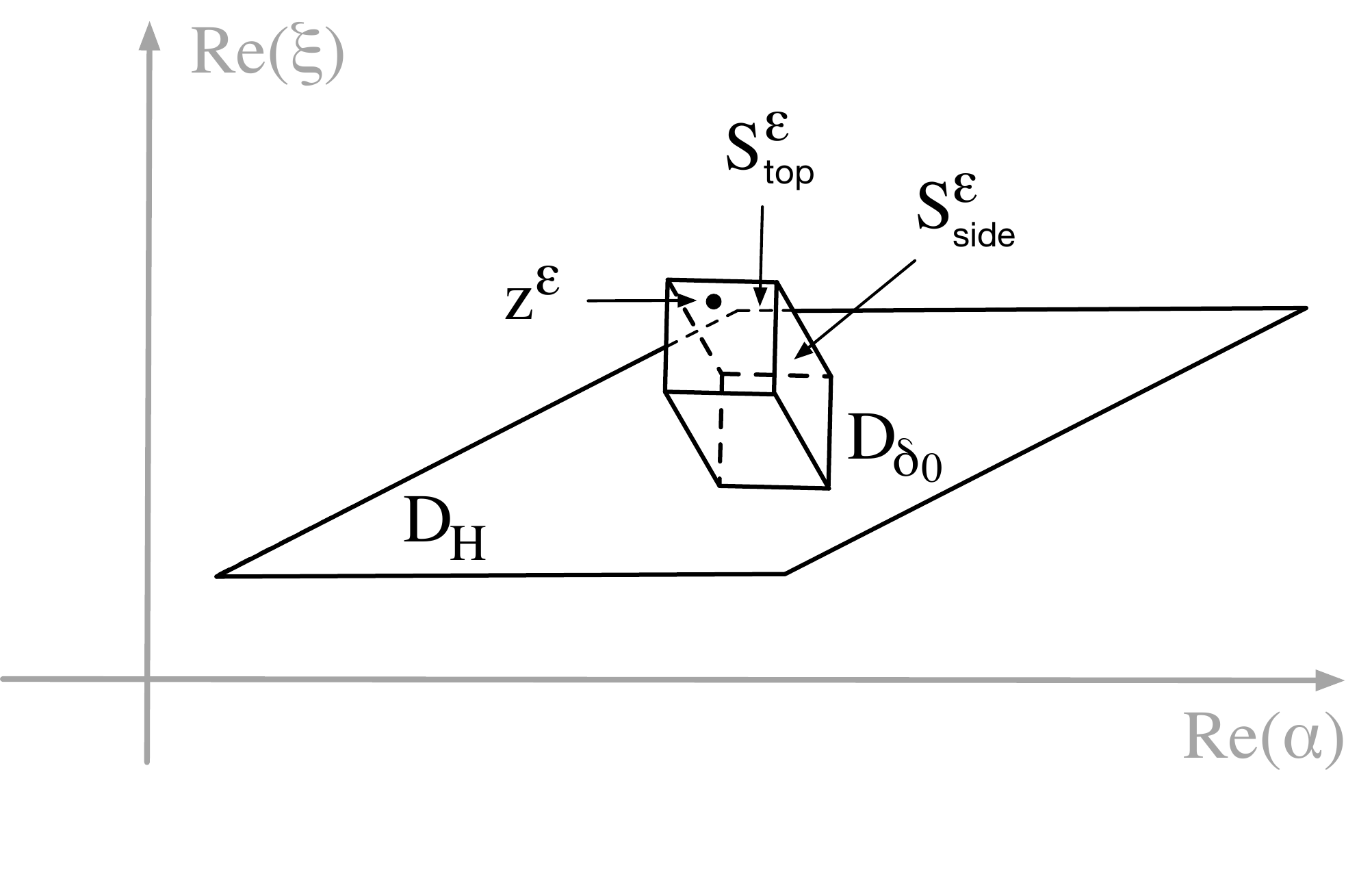}
\caption{The deformed surface $S^{\epsilon}$}
\label{surface} 
\end{figure}

By analyticity, the integral remains the same if we deform the domain from $\mathrm{D_{\delta_0}}$ to $S^{\epsilon}.$

By Proposition \ref{convexity}, $\mathrm{Im}\mathcal W^{\epsilon}$  is concave down on $S^{\epsilon}_{\text{top}}.$ Since $\boldsymbol z^{\epsilon}$ is the critical points of $\mathrm{Im}\mathcal W^{\epsilon},$ it is the only absolute maximum on $S^{\epsilon}_{\text{top}}.$

On the side $S^{\epsilon}_{\text{side}},$ for each $(\boldsymbol\alpha,\boldsymbol\xi)\in \partial \mathrm{D_{\delta_0}},$ consider the function 
$$g^{\epsilon}_{(\boldsymbol\alpha,\boldsymbol\xi)}(t)= \mathrm{Im}\mathcal W^{\epsilon}((\boldsymbol\alpha,\boldsymbol\xi)+t\sqrt{-1}\cdot \mathrm{Im}(\boldsymbol z^{\epsilon}))$$
on $[0,1].$ By Lemma \ref{convexity}, $g^{\epsilon}_{(\boldsymbol\alpha,\boldsymbol\xi)}(t)$ is concave up for any $(\boldsymbol\alpha,\boldsymbol\xi)\in \partial \mathrm{D_{\delta_0}}.$ As a consequence, $g^{\epsilon}_{(\boldsymbol\alpha,\boldsymbol\xi)}(t)\leqslant \max\{g^{\epsilon}_{(\boldsymbol\alpha,\boldsymbol\xi)}(0), g^{\epsilon}_{(\boldsymbol\alpha,\boldsymbol\xi)}(1)\}.$ Now by the previous two steps, since $(\boldsymbol\alpha,\boldsymbol\xi)\in \partial \mathrm{D_{\delta_0}},$ 
$$g^{\epsilon}_{(\boldsymbol\alpha,\boldsymbol\xi)}(0)= \mathrm{Im}\mathcal W^{\epsilon}(\boldsymbol\alpha,\boldsymbol\xi)<\mathrm{Im}\mathcal W^{\epsilon}(\boldsymbol z^{\epsilon});$$
and since $(\boldsymbol\alpha,\boldsymbol\xi)+\sqrt{-1}\cdot \mathrm{Im}(\boldsymbol z^{\epsilon})\in S^{\epsilon}_{\text{top}},$
$$g^{\epsilon}_{(\boldsymbol\alpha,\boldsymbol\xi)}(0)= \mathrm{Im}\mathcal W^{\epsilon}((\boldsymbol\alpha,\boldsymbol\xi)+\sqrt{-1}\cdot \mathrm{Im}(\boldsymbol z^{\epsilon}))<\mathrm{Im}\mathcal W^{\epsilon}(\boldsymbol z^{\epsilon}).$$ As a consequence,
 $$\mathrm{Im}\mathcal W^{\epsilon}(\boldsymbol z^{\epsilon})>\max_{S^{\epsilon}_{\text{side}}} \mathrm{Im}\mathcal W^{\epsilon}.$$

Therefore, we proved that $\boldsymbol z^{\epsilon}$ is the unique maximum point of $\mathrm{Im}\mathcal W^{\epsilon}$ on $S^{\epsilon}\cup\big( \mathrm{D_H}\setminus\mathrm{D_{\delta_0}}\big),$ and $\mathcal W^{\epsilon}$ has critical value $2|T|\pi^2+2\sqrt{-1}\cdot\mathrm{Vol}(M_{E_{\boldsymbol\theta}})$ at $\boldsymbol z^{\epsilon}.$

By Proposition \ref{nonsingular}, $\det\mathrm{Hess}\mathcal W^{\epsilon}(\boldsymbol z^{\epsilon})\neq 0.$


Next, we prove that the domain
$$\big\{(\boldsymbol\alpha,\boldsymbol \xi)\in \overline{\mathrm D_{\delta_0,\mathbb C}}\ \big|\ \mathrm{Im}\mathcal W^{\epsilon}(\boldsymbol\alpha ,\boldsymbol \xi)<\mathrm{Im}\mathcal W^{\epsilon }(\mathbf z^{\epsilon })\big\} $$
deformation retracts to $S_{\text{top}}^{\epsilon }\setminus\{\mathbf z^{\epsilon }\}.$ To see this, for each $\mathbf x \in \mathrm D_{\delta_0},$ let 
$$P_{\mathbf x}=\big\{(\boldsymbol\alpha ,\boldsymbol\xi)\in \mathrm D_{\delta_0,\mathbb C}\ \big|\ \mathrm{Re}(\boldsymbol\alpha ,\boldsymbol \xi)=\mathbf x\big\}$$
and 
$$B_{\mathbf x}=\big\{(\boldsymbol\alpha ,\boldsymbol\xi)\in P_{\mathbf x}\ \big|\ \mathrm{Im}\mathcal W^{\epsilon }(\boldsymbol\alpha ,\boldsymbol \xi)<\mathrm{Im}\mathcal W^{\epsilon }(\mathbf z^{\epsilon })\big\}.$$ Then by Proposition \ref{convexity} that $\mathrm{Im}\mathcal W^{\epsilon }$ is concave up in $\mathrm{Im}(\boldsymbol\alpha ,\boldsymbol \xi),$  $B_{\mathrm{Re}(\mathbf z^{\epsilon })}=\emptyset,$ and $B_{\mathbf x}$ is a non-empty convex subset of $P_{\mathbf x}$ for $\mathbf x\neq \mathrm{Re}(\mathbf z^{\epsilon });$  and by the fact that $\mathbf z^{\epsilon }$ is the unique maximum point of $\mathrm{Im}\mathcal W^{\epsilon }$ on $S^{\epsilon },$ $\mathbf x+\sqrt{-1}\mathrm{Im}(\mathbf z^{\epsilon })\in B_{\mathbf z}$ for $\mathbf x \neq \mathrm{Re}(\mathbf z^{\epsilon }).$ 
 As a consequence, $B_{\mathbf x}$ deformation retracts to $\mathbf x+\sqrt{-1}\mathrm{Im}(\mathbf z^{\epsilon })$ which induces the desired deformation retraction of $\big\{(\boldsymbol\alpha ,\boldsymbol \xi)\in \overline{\mathrm D_{\delta_0,\mathbb C}}\ \big|\ \mathrm{Im}\mathcal W^{\epsilon }(\boldsymbol\alpha ,\boldsymbol \xi)<\mathrm{Im}\mathcal W^{\epsilon }(\mathbf z^{\epsilon })\big\} $ to $S_{\text{top}}^{\epsilon }\setminus\{\mathbf z^{\epsilon }\}.$


Finally,  we estimate the difference between $\mathcal W^{\epsilon}_r$ and $\mathcal W^{\epsilon}.$ By Lemma \ref{converge}, (3), we have
 $$\varphi_r\Big(\frac{\pi}{r}\Big)=\mathrm{Li}_2(1)+\frac{2\pi\sqrt{-1}}{r}\log\Big(\frac{r}{2}\Big)-\frac{\pi^2}{r}+O\Big(\frac{1}{r^2}\Big);$$
and for $z$ with $0<\mathrm{Re z}<\pi$ have
 $$\varphi_r\Big(z+\frac{k\pi}{r}\Big)=\varphi_r(z)+\varphi'_r(z)\cdot\frac{k\pi}{r}+O\Big(\frac{1}{r^2}\Big).$$
 Then by Lemma \ref{converge}, in 
 $$\big\{(\boldsymbol\alpha,\boldsymbol\xi)\in \overline{\mathrm{D_{H,\mathbb C}^\delta}}\ \big|\ |\mathrm{Im}(\alpha_i)| < L\text{ for } i=\{1,\dots,|E|\}, |\mathrm{Im}(\xi_s)| < L\text{ for } s=\{1,\dots,|T|\} \}$$
  for some $L>0,$
 $$\mathcal W^{\epsilon}_r(\boldsymbol\alpha,\boldsymbol\xi)=\mathcal W^{\epsilon}(\boldsymbol\alpha,\boldsymbol\xi)-\frac{4|T|\pi\sqrt{-1}}{r}\log\Big(\frac{r}{2}\Big)+\frac{4\pi \sqrt{-1} \cdot\kappa(\boldsymbol\alpha,\boldsymbol\xi)}{r}+\frac{\nu_r(\boldsymbol\alpha,\boldsymbol\xi)}{r^2},$$
with
 \begin{equation*}
 \begin{split}
&\kappa(\boldsymbol\alpha,\boldsymbol\xi)\\
=&\sum_{s=1}^{|T|}\Big(\frac{1}{2}\sum_{i=1}^4 \sqrt{-1}\tau_{s_i}- \sqrt{-1}\xi_s-\sqrt{-1}\pi-\frac{\sqrt{-1}\pi}{2}\\
&+\frac{1}{4}\sum_{i=1}^4\sum_{j=1}^3\log\big(1-e^{2\sqrt{-1}(\eta_{s_j}-\tau_{s_i})}\big)-\frac{3}{4}\sum_{i=1}^4\log\big(1-e^{2\sqrt{-1}(\tau_{s_i}-\pi)}\big)\\
&+\frac{3}{2}\log\big(1-e^{2\sqrt{-1}(\xi_s-\pi)}\big)-\frac{1}{2}\sum_{i=1}^4\log\big(1-e^{2\sqrt{-1}(\xi_s-\tau_{s_i})}\big)-\frac{1}{2}\sum_{j=1}^3\log\big(1-e^{2\sqrt{-1}(\eta_{s_j}-\xi_s)}\big)\Big)
 \end{split}
 \end{equation*}
 and  $|\nu_r(\boldsymbol\alpha,\boldsymbol\xi)|$ bounded from above by a constant independent of $r.$
 Then
  \begin{equation*}
 \begin{split}
&e^{\sum_{i=1}^{|E|}\epsilon_i\sqrt{-1}\big(\alpha_i+\beta_i+\frac{2\pi}{r}\big)+\frac{r}{4\pi \sqrt{-1}}{\mathcal W}^{\epsilon}_r(\boldsymbol\alpha,\boldsymbol\xi)}\\
=&\Big(\frac{r}{2}\Big)^{-|T|}e^{\sum_{i=1}^{|E|}\epsilon_i\sqrt{-1}(\alpha_i+\beta_i)+\kappa(\boldsymbol\alpha,\boldsymbol\xi)}\cdot e^{\frac{r}{4\pi \sqrt{-1}}\Big(\mathcal W^{\epsilon}(\boldsymbol\alpha,\boldsymbol\xi)+\frac{\nu_r(\alpha,\xi)-\sum_{i=1}^{|E|}\epsilon_i 8\pi^2}{r^2}\Big)}.
 \end{split}
 \end{equation*}

Now let $D_{\mathbf z}=\big\{(\boldsymbol\alpha,\boldsymbol\xi)\in \overline{\mathrm{D_{H,\mathbb C}^\delta}}\ \big|\ |\mathrm{Im}(\alpha_i)| < L\text{ for } i=\{1,\dots,|E|\}, |\mathrm{Im}(\xi_s)| < L\text{ for } s=\{1,\dots,|T|\} \}$ for some $L>0.$ Let $\mathbf a_r=((\beta_1,\dots,\beta_{|E|})$ (recall that $\beta_i=\frac{2\pi b^{(r)}_i}{r}$ depends on $r$),
$f^{\mathbf a_r}(\boldsymbol\alpha,\boldsymbol\xi)=\frac{\mathcal W^{\epsilon}(\boldsymbol\alpha,\boldsymbol\xi)}{4\pi\sqrt{-1}},$ $g^{\mathbf a_r}(\boldsymbol\alpha,\boldsymbol\xi)=\psi(\boldsymbol\alpha,\boldsymbol\xi)e^{\sum_{i=1}^{|E|}\epsilon_i\sqrt{-1}(\alpha_i+\beta_i)+\kappa(\boldsymbol\alpha,\boldsymbol\xi)},$ $f_r^{\mathbf a_r}(\boldsymbol\alpha,\boldsymbol\xi)=\frac{{\mathcal W}_r^{\epsilon}(\boldsymbol\alpha,\boldsymbol\xi)}{4\pi\sqrt{-1}}-\frac{|T|}{r}\log\big(\frac{r}{2}\big),$ $\upsilon_r(\boldsymbol\alpha,\boldsymbol\xi)=\nu_r(\boldsymbol\alpha,\boldsymbol\xi)-\sum_{i=1}^{|E|}\epsilon_i 8\pi^2,$ $S_r=S^{\epsilon}\cup \big(\mathrm{D_H}\setminus\mathrm{D_{\delta_0}}\big)$ and $\boldsymbol z^{\epsilon}$ is the critical point of $f$ in $D_{\mathbf z}.$  Then all the conditions of Proposition \ref{saddle} are satisfied and the result follows. 

When $\beta_1=\dots=\beta_{|E|}=\pi,$ a direct computation shows that 
\begin{equation*}
\begin{split}
C^{\epsilon}(\boldsymbol z^{\epsilon})=\frac{r^{|E|+|T|}}{2^{2|E|+|T|}\pi^{|E|+|T|}}\Big(\frac{2\pi}{r}\Big)^{
\frac{|E|+|T|}{2}}\Big(\frac{r}{2}\Big)^{-|T|} g\Big(\pi,\dots,\pi,\frac{7\pi}{4},\dots,\frac{7\pi}{4}\Big)=\frac{(-1)^{|T|}r^{\frac{|E|-|T|}{2}}}{2^{\frac{3|E|+|T|}{2}}\pi^{\frac{|E|+|T|}{2}}}.
\end{split}
\end{equation*}
\end{proof}

\begin{corollary}\label{5.8} If $\epsilon>0$ is sufficiently small and all $\{\beta_1,\dots,\beta_{|E|}\}$ are in $\{\pi-\epsilon,\pi+\epsilon
\},$ then 
$$\sum_{\epsilon\in\{1,-1\}^E}\frac{C^{\epsilon}(\boldsymbol z^{\epsilon})}{\sqrt{-\det\mathrm{Hess}\Big(\frac{\mathcal W^{\epsilon}(\boldsymbol z^{\epsilon})}{4\pi\sqrt{-1}}\Big)}}\neq 0.$$
\end{corollary}

\begin{proof} If $\beta_i=\dots=\beta_{|E|}=\pi,$ then all $z^{\epsilon}=\big(\pi,\dots,\pi,\frac{7\pi}{4},\dots,\frac{7\pi}{4}\big)$ and all $\mathcal W^{\epsilon}$ are the same functions. As a consequence, all the $C^{\epsilon}(\boldsymbol z^{\epsilon})$'s and all Hessian determinants $\det\mathrm{Hess}\Big(\frac{\mathcal W^{\epsilon}(\boldsymbol z^{\epsilon})}{4\pi \sqrt{-1}}\Big)$'s are the same at this point, imply that the sum is not equal to zero. Then by continuity, if $\epsilon$ is small enough, then the sum remains nonzero.
\end{proof}

\begin{remark} In \cite{WY3, WY4}, we prove that all $C^{\epsilon}(\boldsymbol z^{\epsilon})$'s and all $\det\mathrm{Hess}\mathcal W^{\epsilon}(\boldsymbol z^{\epsilon})$'s are always the same for any given $\{\beta_1,\dots,\beta_{|E|}\},$ and are closely related to the adjoint twisted Reidemeister torsion of $M_{E_{\boldsymbol\theta}}.$
\end{remark}


\subsection{Estimate of the other Fourier coefficients}\label{ot}

\begin{proposition}\label{other} Suppose $\{\beta_1,\dots,\beta_{|E|}\}$ are in $\{\pi-\epsilon,\pi+\epsilon
\}$ for a sufficiently small $\epsilon>0.$  If $(\mathbf m, \mathbf n)\neq(0,\dots,0),$ then
$$\Big|\widehat{f^{\epsilon}_r}(\mathbf m, \mathbf n)\Big|<O\Big(e^{\frac{r}{2\pi}\big(\mathrm{Vol}(M_{E_\theta})-\epsilon'\big)}\Big)$$
for some $\epsilon'>0.$
\end{proposition}

\begin{proof} Recall that if $\beta_1=\dots=\beta_{|E|}=\pi,$ then the total derivative 
$$D\mathcal W^{\epsilon}\Big(\pi,\dots,\pi,\frac{7\pi}{4},\dots,\frac{7\pi}{4}\Big)=(0,\dots,0).$$
Hence there exists a $\delta_1>0$ and an $\epsilon>0$ such that if $\{\beta_1,\dots,\beta_{|E|}\}$   are in $\{\pi-\epsilon,\pi+\epsilon
\},$ then for all $(\boldsymbol\alpha,\boldsymbol\xi)\in D_{\delta_1,\mathbb C}$ and for any unit vector $\mathbf u=(u_1,\dots,u_{|E|},w_1,\dots,w_{|T|})\in \mathbb R^{|E|+|T|},$ the directional derivatives 
$$|D_{\mathbf u}\mathrm{Im}\mathcal W^{\epsilon}(\boldsymbol\alpha,\boldsymbol\xi)|=\bigg|\sum_{i=1}^{|E|}u_i\frac{\partial \mathrm{Im}\mathcal W^{\epsilon}}{\partial \mathrm{Im}(\alpha_i)}+\sum_{s=1}^{|T|}w_s\frac{\partial \mathrm{Im}\mathcal W^{\epsilon}}{\partial \mathrm{Im}(\xi_s)}\bigg|<\frac{2\pi-\epsilon''}{2\sqrt{2|E|+2|T|}}$$
for some $\epsilon''>0.$

On $\mathrm{D_H},$ we have 
\begin{equation*}
\begin{split}
 &\mathrm{Im}\Big(\mathcal W^{\epsilon}(\boldsymbol\alpha,\boldsymbol\xi)-\sum_{i=1}^{|E|}2\pi m_i\alpha_i-\sum_{s=1}^{|T|}4\pi n_s\xi_s\Big)=\mathrm{Im}\mathcal W^{\epsilon}(\boldsymbol\alpha,\boldsymbol\xi).
\end{split}
\end{equation*}
Then by Lemma \ref{convexity}, Proposition \ref{absm} and the compactness of $\mathrm{D_H}\setminus\mathrm{D_{\delta_1}},$
$$2|T|v_8>\max_{\mathrm{D_H}\setminus\mathrm{D_{\delta_1}}} \mathrm{Im}\Big(\mathcal W^{\epsilon}(\boldsymbol\alpha,\boldsymbol\xi)-\sum_{i=1}^{|E|}2\pi m_i\alpha_i-\sum_{s=1}^{|T|}4\pi n_s\xi_s\Big)+\epsilon'''$$
for some $\epsilon'''>0.$
By Proposition \ref{crit} and continuity, if $\{\beta_1,\dots,\beta_{|E|}\}$   are sufficiently close to $\pi,$ then the critical point $\boldsymbol z^{\epsilon}$ of $\mathcal W^{\epsilon}$  as in Proposition \ref{crit} lies in $\mathrm{D_{\delta_1,\mathbb C}},$ and $\mathrm{Im}\mathcal W^{\epsilon}(\boldsymbol z^{\epsilon})=2\mathrm{Vol}(M_{E_{\boldsymbol\theta}})$ is sufficiently close to $2|T|v_8$ so that
\begin{equation}\label{b}
\mathrm{Im}\mathcal W^{\epsilon}(\boldsymbol z^{\epsilon})>\max_{\mathrm{D_H}\setminus\mathrm{D_{\delta_1}}} \mathrm{Im}\Big(\mathcal W^{\epsilon}(\boldsymbol\alpha,\boldsymbol\xi)-\sum_{i=1}^{|E|}2\pi m_i\alpha_i-\sum_{s=1}^{|T|}4\pi n_s\xi_s\Big)+\epsilon'''.
\end{equation}
 
Therefore, we only need to estimate  the integral on $\mathrm{D_{\delta_1}}.$ 

If $(\mathbf m,\mathbf n)\neq (0,\dots,0),$ then there is at least one element  of $\{m_1,\dots,m_{|E|}\}$ or of $\{n_1,\dots,n_{|T|}\}$  that is nonzero. Without loss of generality, assume that $m_1\neq 0,$ as the calculation for the case that $n_1\neq 0$ can be carried out in exactly the same manner.

If $m_1>0,$ then consider the surface $S^+=S^+_{\text{top}}\cup S^+_{\text{side}}$ in $\overline{\mathrm{D_{\delta_1,\mathbb C}}}$ where
$$S^+_{\text{top}}=\{ (\boldsymbol\alpha,\boldsymbol\xi)\in \mathrm{D_{\delta_1,\mathbb C}}\ |\ (\mathrm{Im}(\boldsymbol\alpha),\mathrm{Im}(\boldsymbol\xi))=(\delta_1,0,\dots,0)\}$$
and
$$S^+_{\text{side}}=\{ (\boldsymbol\alpha,\boldsymbol\xi)+(t\sqrt{-1}\delta_1,0,\dots,0)\ |\ (\boldsymbol\alpha,\boldsymbol\xi)\in\partial \mathrm{D_{\delta_1}}, t\in[0,1]\}.$$      
On the top, for any $(\boldsymbol\alpha,\boldsymbol\xi)\in S^+_{\text{top}},$ by the Mean Value Theorem, 
\begin{equation*}
\begin{split}
\big|\mathrm{Im}\mathcal W^{\epsilon}(\boldsymbol z^{\epsilon})-\mathrm{Im}\mathcal W^{\epsilon}(\boldsymbol\alpha,\boldsymbol\xi)\big|
=&\big|D_{\mathbf u}\mathrm{Im}\mathcal W^{\epsilon}(\boldsymbol z)\big|\cdot\big\|\boldsymbol z^{\epsilon}-(\boldsymbol\alpha,\boldsymbol\xi)\big\|\\
<&\frac{2\pi-\epsilon''}{2\sqrt{2|E|+2|T|}}\cdot2\sqrt{2|E|+2|T|} \delta_1\\
=&2\pi\delta_1-\epsilon''\delta_1,
\end{split}
\end{equation*}
where $\boldsymbol z$ is some point on the line segment connecting $\boldsymbol z^{\epsilon}$ and $(\boldsymbol\alpha,\boldsymbol\xi),$ $\mathbf u=\frac{\boldsymbol z^{\epsilon}-(\boldsymbol\alpha,\boldsymbol\xi)}{\|\boldsymbol z^{\epsilon}-(\boldsymbol\alpha,\boldsymbol\xi)\|}$ and $2\sqrt{2|E|+2|T|} \delta_1$ is the diameter of $\mathrm{D_{\delta_1,\mathbb C}}.$
Then
\begin{equation*}
\begin{split}
\mathrm{Im}\Big(\mathcal W^{\epsilon}(\boldsymbol\alpha,\boldsymbol\xi)-\sum_{i=1}^{|E|}2\pi m_i\alpha_i-\sum_{s=1}^{|T|}4\pi n_s\xi_s\Big)=&\mathrm{Im}\mathcal W^{\epsilon}(\boldsymbol\alpha,\boldsymbol\xi)-2\pi m_1\delta_1\\
<&\mathrm{Im}\mathcal W^{\epsilon}(\boldsymbol z^{\epsilon})+2\pi\delta_1-\epsilon''\delta_1-2\pi\delta_1\\
=&\mathrm{Im}\mathcal W^{\epsilon}(\boldsymbol z^{\epsilon})-\epsilon'' \delta_1.
\end{split}
\end{equation*}

On the side, for any point $(\boldsymbol\alpha,\boldsymbol\xi)+(t\sqrt{-1}\delta_1,0,\dots,0)\in S^+_{\text{side}},$ by the Mean Value Theorem again, we have
$$\big|\mathrm{Im}\mathcal W^{\epsilon}\big((\boldsymbol\alpha,\boldsymbol\xi)+(t\sqrt{-1}\delta_1,0,\dots,0)\big)-\mathrm{Im}\mathcal W^{\epsilon}(\boldsymbol\alpha,\boldsymbol\xi)\big|<\frac{2\pi-\epsilon''}{2\sqrt{2|E|+2|T|}} t\delta_1.$$
Then 
\begin{equation*}
\begin{split}
\mathrm{Im}\mathcal W^{\epsilon}\big((\boldsymbol\alpha,\boldsymbol\xi)+(t\sqrt{-1}\delta_1,0,\dots,0)\big)-2\pi m_1 t\delta_1<&\mathrm{Im}\mathcal W^{\epsilon}(\boldsymbol\alpha,\boldsymbol\xi)+\frac{2\pi-\epsilon''}{2\sqrt{2|E|+2|T|}} t\delta_1-2\pi t\delta_1\\
<&\mathrm{Im}\mathcal W^{\epsilon}(\boldsymbol\alpha,\boldsymbol\xi)\\
<&\mathrm{Im}\mathcal W^{\epsilon}(\boldsymbol z^{\epsilon})-\epsilon''',
\end{split}
\end{equation*}
where the last inequality comes from that $(\boldsymbol\alpha,\boldsymbol\xi)\in \partial \mathrm{D_{\delta_1}}\subset \mathrm{D_H}\setminus\mathrm{D_{\delta_1}}$ and (\ref{b}).

Now let $\epsilon'=\min\{\epsilon''\delta_1,\epsilon'''\},$ then on $S^+\cup \big(\mathrm{D_H}\setminus\mathrm{D_{\delta_1}}\big),$ 
$$\mathrm{Im}\Big(\mathcal W^{\epsilon}(\boldsymbol\alpha,\boldsymbol\xi)-\sum_{i=1}^{|E|}2\pi m_i\alpha_i-\sum_{s=1}^{|T|}4\pi n_s\xi_s\Big)<\mathrm{Im}\mathcal W^{\epsilon}(z^{\epsilon})-\epsilon',$$
and the result follows.

If $m_1<0,$ then we consider the surface $S^-=S^-_{\text{top}}\cup S^-_{\text{side}}$ in $\overline{\mathrm{D_{\delta_1,\mathbb C}}}$ where
$$S^-_{\text{top}}=\{ (\boldsymbol\alpha,\boldsymbol\xi)\in \mathrm{D_{\delta_1,\mathbb C}}\ |\ (\mathrm{Im}(\boldsymbol\alpha),\mathrm{Im}(\boldsymbol\xi))=(-\delta_1,0,\dots,0)\}$$
and
$$S^-_{\text{side}}=\{ (\boldsymbol\alpha,\boldsymbol\xi)-(t\sqrt{-1}\delta_1,0,\dots,0)\ |\ (\boldsymbol\alpha,\boldsymbol\xi)\in\partial \mathrm{D_{\delta_1}}, t\in[0,1]\}.$$      
Then the same estimate as in the previous case proves that on
$S^-\cup \big(\mathrm{D_H}\setminus\mathrm{D_{\delta_1}}\big),$ 
$$\mathrm{Im}\Big(\mathcal W^{\epsilon}(\boldsymbol\alpha,\boldsymbol\xi)-\sum_{i=1}^{|E|}2\pi m_i\alpha_i-\sum_{s=1}^{|T|}4\pi n_s\xi_s\Big)<\mathrm{Im}\mathcal W^{\epsilon}(\boldsymbol z^{\epsilon})-\epsilon',$$
from which the result follows.
\end{proof}


\subsection{Estimate of the error term}\label{ee}

The goal of this section is to estimate the error term in Proposition \ref{Poisson}.

\begin{proposition}\label{error} The error term in Proposition \ref{Poisson} is less than $O\big(e^{\frac{r}{2\pi}(\mathrm{Vol}(M_{E_{\boldsymbol\theta}})-\epsilon')}\big)$
for some $\epsilon'>0.$
\end{proposition}

For the proof we need the following estimate, which first appeared in \cite[Proposition 8.2]{GL} for $q=e^{\frac{\pi \sqrt{-1}}{r}},$ and for the root $q=e^{\frac{2\pi \sqrt{-1}}{r}}$ in \cite[Proposition 4.1]{DK}.

\begin{lemma}\label{est}
 For any integer $0<n<r$ and at $q=e^{\frac{2\pi \sqrt{-1}}{r}},$
 $$ \log\left|\{n\}!\right|=-\frac{r}{2\pi}\Lambda\left(\frac{2n\pi}{r}\right)+O\left(\log(r)\right).$$
\end{lemma}

\begin{proof}[Proof of Proposition \ref{error} ]
Let 
\begin{equation*}
M=\max\Big\{\sum_{s=1}^{|T|}2V(\alpha_{s_1},\dots,\alpha_{s_6},\xi_s)\ \Big|\ (\boldsymbol\alpha,\boldsymbol\xi)\in\partial \mathrm {D_H}\cup\big(\mathrm {D_A}\setminus \mathrm{D_H}\big)\Big\}
\end{equation*}
Then by \cite[Section 4]{BDKY},
$$M<2|T|v_8=2\mathrm{Vol}(M_{E_{(0,\dots,0)}});$$
and by continuity, if $\boldsymbol\theta$ is sufficiently closed to $(0,\dots,0),$ then
$$M<2\mathrm{Vol}(M_{E_{\boldsymbol\theta}}).$$

Now by Lemma \ref{est} and the continuity, for $\epsilon'=\frac{2\mathrm{Vol}(M_{E_{\boldsymbol\theta}})-M}{3},$ we can choose a sufficiently small $\delta>0$ so that if $\big(\frac{2\pi \mathbf a}{r},\frac{2\pi \mathbf k}{r}\big)\notin \mathrm{D_H^\delta},$ then
$$\Big|g_r^{\epsilon}(\mathbf a, \mathbf k)\Big|<O\Big(e^{\frac{r}{4\pi}(M+\epsilon')}\Big)=O\Big(e^{\frac{r}{2\pi}(\mathrm{Vol}(M_{E_{\boldsymbol\theta}})-\epsilon')}\Big).$$
Let $\psi$ be the bump function supported on $(\mathrm{D_H}, \mathrm{D_H^{\delta}}).$ Then the error term in Proposition \ref{Poisson} is less than $O\big(e^{\frac{r}{2\pi}(\mathrm{Vol}(M_{E_{\boldsymbol\theta}})-\epsilon')}\big).$
\end{proof}


\subsection{Proof of Theorem \ref{main}}\label{pf}

\begin{proof}[Proof of Theorem \ref{main}] Let $\epsilon>0$ be sufficiently small so that the conditions of Propositions \ref{critical}, \ref{other} and  \ref{error} and of Corollary \ref{5.8} are satisfied, and suppose $\{\beta_1,\dots,\beta_{|E|}\}$ are all in $(\pi-\epsilon, \pi+\epsilon).$

By Propositions \ref{4.2}, \ref{Poisson}, \ref{critical}, \ref{other} and  \ref{error}, 
\begin{equation*}
\begin{split}
&\mathrm {TV}_r(M,E,\mathbf b)\\
=&\frac{(-1)^{|E|\big(\frac{r}{2}+1\big)}2^{\mathrm{rank H}_2(M;\mathbb Z_2)-|T|}}{\{1\}^{|E|-|T|}}\Big(\sum_{\epsilon\in\{1,-1\}^{|E|}}\widehat{ f_r^{\epsilon}}(0,\dots,0)\Big)\Big(1+O\big(e^{\frac{r}{2\pi}{(-\epsilon')}}\big)\Big)\\
=&\frac{(-1)^{|E|\big(\frac{r}{2}+1\big)}2^{\mathrm{rank H}_2(M;\mathbb Z_2)-|T|}}{\{1\}^{|E|-|T|}}\bigg( \sum_{\epsilon\in\{1,-1\}^{|E|}}\frac{C^{\epsilon}(z^{\epsilon})}{\sqrt{-\det\mathrm{Hess}\Big(\frac{\mathcal W^{\epsilon}(\boldsymbol z^{\epsilon})}{4\pi\sqrt{-1}}\Big)}}\bigg)  e^{\frac{r}{2\pi}\mathrm{Vol}(M_{E_{\boldsymbol\theta}})}\Big( 1 + O \Big( \frac{1}{r} \Big) \Big);
\end{split}
\end{equation*}
and by Corollary \ref{5.8},
 $$\sum_{\epsilon\in\{1,-1\}^{|E|}}\frac{C^{\epsilon}(z^{\epsilon})}{\sqrt{-\det\mathrm{Hess}\Big(\frac{\mathcal W^{\epsilon}(\boldsymbol z^{\epsilon})}{4\pi\sqrt{-1}}\Big)}}\neq 0,$$
which completes the proof.
\end{proof}


\noindent
Tian Yang\\
Department of Mathematics\\  Texas A\&M University\\
College Station, TX 77843, USA\\
(tianyang@math.tamu.edu)

\end{document}